\documentclass[11pt,letterpaper]{article}
\usepackage{paper-arxiv}
\usepackage{paper-common}
\begin{document}

\title{Dimension-independent eigenvalue approximation for sparse matrices}
\titlemark{Dimension-independent eigenvalue approximation}

\emsauthor*{1}{\givenname{Natalia}\surname{Haj\l{}asz}\mrid{}\zblid{}\orcid{}}{N.~Haj\l{}asz}
\emsauthor{2}{\givenname{Christopher}\surname{Musco}\mrid{}\zblid{}\orcid{}}{C.~Musco}
\emsauthor{3}{\givenname{Jonathan}\surname{Weare}\mrid{}\zblid{}\orcid{}}{J.~Weare}
\emsauthor{4}{\givenname{Robert J.}\surname{Webber}\mrid{}\zblid{}\orcid{}}{R.~J.~Webber}

\Emsaffil{1}{
\department{Courant Institute of Mathematical Sciences}
\organisation{New York University}
\address{251 Mercer Street}
\zip{NY 10012}
\city{New York}
\country{USA}
\affemail{nh2429@nyu.edu}}
\Emsaffil{2}{
\department{Department of Computer Science and Engineering, Tandon School of Engineering}
\organisation{New York University}
\address{370 Jay Street}
\zip{NY 11201}
\city{Brooklyn}
\country{USA}
\affemail{cmusco@nyu.edu}}
\Emsaffil{3}{
\department{Courant Institute of Mathematical Sciences}
\organisation{New York University}
\address{251 Mercer Street}
\zip{NY 10012}
\city{New York}
\country{USA}
\affemail{weare@nyu.edu}}
\Emsaffil{4}{
\department{Department of Mathematics}
\organisation{University of California San Diego}
\address{9500 Gilman Drive}
\zip{CA 92093-0112}
\city{La Jolla}
\country{USA}
\affemail{rwebber@ucsd.edu}}

\newcommand{\chris}[1]{\textcolor{blue}{{\bfseries [Chris: }{#1}{\bfseries ]}}}

\classification[65F50]{65F15}

\begin{abstract}
Many applications in science and engineering require solving sparse Hermitian eigenvalue problems that are so large that even algorithms whose runtime scales linearly with the matrix dimension are unacceptably slow. 
Responding to this challenge, there is a growing body of work on \emph{sparsified iterative methods} and \emph{submatrix selection methods} that often perform well in practice. 
However, theoretical understanding of these methods is limited. 

In this paper, we take a step toward better understanding dimension-independent eigenvalue computation by analyzing the sparsified power method.
We prove that the method approximates a leading eigenpair to a prescribed accuracy in a number of iterations comparable to the standard power method.
The required number of retained nonzeros depends explicitly on the target accuracy, the initial overlap with the leading eigenvector, the compressibility of the leading eigenvector, and the relative spectral gap.
When these quantities and the column sparsity are independent of dimension, the runtime is independent of dimension.

In contrast, we show that a comparable worst-case guarantee cannot hold for a prototypical sparsified submatrix selection method. 
Submatrix selection is the approach of choice in many quantum chemistry applications.
However, we show that, even with substantial initial eigenvector overlap, a single iteration of the method can produce a vector that is orthogonal to the leading eigenvector, after which the method cannot recover.
Nevertheless, experiments on matrices from the SuiteSparse matrix collection and quantum chemistry problems illustrate the practical effectiveness of submatrix selection methods despite these worst-case limitations.
\end{abstract}

\maketitle

\section{Introduction}
\label{sec:introduction}

Krylov subspace methods are the standard tools for computing the largest-magnitude eigenvalues of massive sparse matrices \cite{Saad:2011}.
For an $n\times n$ matrix $\bm{A}$ with $s$ nonzeros per column, standard implementations of these methods require $\mathcal{O}(sn)$ time per iteration and can converge quickly to outlying eigenvalues.
Yet even a linear dependence on the matrix dimension $n$ can become prohibitive for applications in network science \cite{ChingEdunovKabiljo:2015}, machine learning \cite{GhorbaniKrishnanXiao:2019}, and computational chemistry \cite{HolmesTubmanUmrigar:2016}.
These challenging applications raise the question:
\begin{quote}
\emph{Can we approximate a largest-magnitude eigenvalue and an associated eigenvector of $\bm{A}$ in time independent of the matrix dimension $n$?}
\end{quote}
Without additional structure, the answer to this question is no.
Even for a positive-semi\-definite diagonal matrix, a single large entry can be hidden among $n$ diagonal locations.
Finding this entry, and hence approximating the largest-magnitude eigenvalue, requires $\Omega(n)$ time in the worst case.

Nevertheless, in fields where the solution of massive eigenvalue problems is required, practitioners seem undeterred by such examples.
In quantum chemistry, for example, it is common to approximate the outlying eigenvalues of implicit
sparse Hermitian matrices with dimension $n=10^{25}$ or higher while accessing only a tiny fraction of the matrix entries  \cite{Li2018shci,Levine2020dasscfasci,Eriksen2020benzene, GreeneWebberSmith:2022,Christlmaier2022fciqmc}.
Two well-established approaches are sparsified iterative methods, which alternate between matrix-vector multiplication and sparsification to a small number of significant entries \cite{YuanZhang:2013,LimWeare:2017,GreeneWebberWeare:2019}, and submatrix selection methods, which adaptively select a small set of indices and diagonalize the corresponding principal submatrix \cite{HuronMalrieuRancurel:1973,Greer:1995mcci,TubmanLeeTakeshita:2016,HolmesTubmanUmrigar:2016}.
In these methods, the matrix is never formed explicitly; instead, its sparse columns are generated as needed.
However, rigorous theory for these methods is scarce, despite their practical success.
We ask: What structural factors govern their convergence or non-convergence for the Hermitian eigenvalue problem?

Our results identify three governing quantities: the overlap of the initial vector with the leading eigenvector, the relative spectral gap, and the compressibility of the leading eigenvector.
We make the following contributions:
\begin{enumerate}
\item A direct implementation of the power method can become exponentially more expensive at each step as the iterates become dense.
We show that enforcing sparsity preserves a convergence rate comparable to that of the standard power method.
We give an explicit sufficient sparsity level in terms of the target accuracy, the initial overlap, the relative spectral gap, and the compressibility of the leading eigenvector.
\item We identify a theoretical limitation for a sparsified submatrix selection method inspired by selected configuration interaction methods in quantum chemistry \cite{Greer:1995mcci,TubmanLeeTakeshita:2016,SchriberEvangelista:2016}.
For any nontrivial relative spectral gap, we construct instances in which the method loses all overlap with the leading eigenvector after a single iteration and never recovers.
For example, with an initial squared overlap of $10\%$, this failure can occur even when the relative spectral gap is as large as $0.45$.
\end{enumerate}

We assume that any requested column of $\bm{A}$ can be generated in time proportional to the number of nonzeros in the column.
Under this access model, our runtime bound for the truncated power method is independent of $n$ when the three governing quantities above, the target accuracy, the column sparsity, and the sparsity of the initial vector are fixed.
Our chemistry experiments show that sparsified submatrix selection can reach a given accuracy in fewer outer iterations than the truncated power method, yet we reveal it is vulnerable to complete loss of overlap in the worst case.

\subsection{Organization and notation}

The rest of the paper is organized as follows.
\Cref{sec:main} describes preliminary observations, main results, and related work, \cref{sec:experiments_chemistry} presents quantum chemistry experiments, \cref{sec:analysis_classical,sec:analysis_sparsified_power,sec:analysis_sparsified_submatrix} present mathematical analysis, and \cref{sec:outlook} concludes.

For a Hermitian matrix $\bm{A}$, we call an eigenvalue of maximum magnitude a leading eigenvalue and an associated eigenvector a leading eigenvector.
We order the eigenvalues by magnitude, $|\lambda_1|\geq|\lambda_2|\geq\cdots\geq|\lambda_n|$, breaking ties in magnitude arbitrarily.
We let $\bm{v}_1,\ldots,\bm{v}_n$ denote a corresponding system of orthonormal eigenvectors.
The overlap of a unit-norm vector $\bm{x}$ with $\bm{v}_1$ refers to the quantity $|\bm{v}_1^* \bm{x}|$.
The notation $\operatorname{nnz}(\bm{A})$ counts the number of nonzero entries in a matrix $\bm{A}$, and we use the same notation for vectors.

We often work with submatrices and subvectors.
The notation $\bm{A}(\mathsf{S},\mathsf{T})$ denotes the submatrix of $\bm{A}$ with rows indexed by $\mathsf{S}$ and columns indexed by $\mathsf{T}$, while $\bm{A}(\mathord{\cdot},\mathsf{T})$ denotes the submatrix containing the columns indexed by $\mathsf{T}$.
Similarly, $\bm{x}(\mathsf{T})$ denotes the restriction of $\bm{x}$ to the indices in $\mathsf{T}$.
For an index set $\mathsf{T}$, $\mathbf{I}(\mathord{\cdot},\mathsf{T})$ denotes the columns of the identity matrix indexed by $\mathsf{T}$.
Thus, $\mathbf{I}(\mathord{\cdot},\mathsf{T})\bm{y}$ embeds $\bm{y}$ into the ambient space by placing its entries at the indices in $\mathsf{T}$ and setting all other entries to zero.
Likewise, $\mathbf{I}(\mathord{\cdot},\mathsf{S})\bm{M}\mathbf{I}(\mathsf{S},\mathord{\cdot})$ embeds $\bm{M}$ as the principal submatrix indexed by $\mathsf{S}$, with zeros elsewhere.

\section{Preliminary observations, main results, and related work}
\label{sec:main}

This section offers preliminary observations on classical iterative methods, presents our main results for sparsified iterative methods, and discusses related work.

\subsection{Preliminary observations on classical iterative methods} \label{sec:prelim_results}

Our work starts with a simple observation.
If we can construct a sparse unit-norm vector $\bm{x}_0$ that has substantial overlap with a leading eigenvector, then we can estimate the magnitude of the leading eigenvalue for a sparse Hermitian matrix in dimension-independent time.
The following proposition, which is proved in \cref{sec:analysis_classical}, makes this claim precise.

\begin{proposition}[Dimension-independent convergence]
\label{prop:informal_introduction}
Suppose $\bm{A}\in \mathbb{C}^{n \times n}$ is a Hermitian matrix with a leading eigenpair $(\lambda_1, \bm{v}_1)$ and at most $s$ nonzero entries per row/column.
Assume knowledge of a unit-norm vector $\bm{x}_0 \in \mathbb{C}^{n}$ with $\bm{v}_1^* \bm{x}_0 \neq 0$.
For any $\varepsilon \in (0, 1)$, define
\begin{equation*}
    C = \operatorname{nnz}(\bm{x}_0) (s+1)^q, \qquad \text{where } q = \left\lceil\frac{1}{\sqrt{8 \varepsilon}}\log\biggl(\frac{4}{|\bm{v}_1^* \bm{x}_0|^2}\biggr)\right\rceil + 1.
\end{equation*}
Then we can compute an eigenvalue estimate satisfying $(1 - \varepsilon) |\lambda_1| \leq |\hat{\lambda}_1| \leq |\lambda_1|$ by accessing at most $C$ nonzero entries and performing $\mathcal{O}(C)$ additional arithmetic operations.
\end{proposition}
Importantly, the complexity bound does not depend explicitly on the matrix dimension $n$.
For fixed $s$, $\operatorname{nnz}(\bm{x}_0)$, $\varepsilon$, and initial overlap $|\bm{v}_1^* \bm{x}_0|$, its cost is independent of $n$.
In this result, no spectral gap is required.

\Cref{prop:informal_introduction}  is not hard to prove.
If $\bm{A}$ is positive-semidefinite, we can obtain a similar result with a worse dependence on $\varepsilon$ by running the power method initialized from $\bm{x}_0$.
Assuming $|\bm{v}_1^* \bm{x}_0|$ is bounded from below independently of $n$, the power method produces the desired eigenvalue estimate in a number of steps that is independent of $n$ \cite{OLearyStewartVandergraft:1979,KuczynskiWozniakowski:1992}.
At every iteration, the number of accessed entries grows by a factor of at most $s$.
The runtime grows by the same factor, since the sparse matrix-vector product dominates the cost.

We obtain \cref{prop:informal_introduction} from a more refined argument in which we run the Lanczos method starting from $\bm{x}_0$.
We can also improve the dependence on $q$ for many natural sparsity patterns.
For banded matrices with fixed bandwidth, the number of distinct entries accessed grows only linearly with $q$, rather than exponentially.

\begin{algorithm}[t]
\caption[Submatrix selection method]{Submatrix selection method}
\label{alg:greedy}
\raggedright
{\bfseries Input}: Hermitian matrix $\bm{A} \in \mathbb{C}^{n \times n}$.
Starting index set $\mathsf{S}_0 \subseteq \{1, \ldots, n\}$.
Number of iterations \(q\).
\par {\bfseries Output}: Approximation $(\hat{\lambda}_1, \hat{\bm{v}}_1)$ to a leading eigenpair of $\bm{A}$.\par
\begin{algorithmic}[1]
\FOR{$t=1, \ldots, q$}
\STATE $\mathsf{S}_t \gets \mathsf{S}_{t-1} \cup \{k: \bm{A}(k,j) \neq 0 \text{ for some } j \in \mathsf{S}_{t-1}\}$
\ENDFOR
\STATE {\bfseries Return}: $(\hat{\lambda}_1, \hat{\bm{v}}_1) \gets$ a leading eigenpair of $\mathbf{I}(\cdot, \mathsf{S}_q) \bm{A}(\mathsf{S}_q, \mathsf{S}_q) \mathbf{I}(\mathsf{S}_q, \cdot)$.
\end{algorithmic}
\end{algorithm}

The power and Lanczos methods use the accessed matrix entries to evaluate a sequence of matrix-vector products. 
In contrast, methods for large eigenvalue problems in quantum chemistry use these entries more exhaustively, optimizing the eigenvector approximation over an adaptively chosen set of indices \cite{HuronMalrieuRancurel:1973,Greer:1995mcci,TubmanLeeTakeshita:2016, SchriberEvangelista:2016,HolmesTubmanUmrigar:2016,GarnironApplencourtGasperich:2019}. As a prototypical example of this approach, the submatrix selection method in \cref{alg:greedy} iteratively grows a set of indices $\mathsf{S}_q$ connected by nonzero entries of $\bm{A}$ and diagonalizes the entire selected principal submatrix.
When $\mathsf{S}_0 = \operatorname{supp}(\bm{x}_0)$, the approximation space for submatrix selection contains the Krylov subspace
\begin{equation*}
    \operatorname{span}\{\bm{x}_0,\bm{A} \bm{x}_0,\ldots,\bm{A}^q\bm{x}_0\}.
\end{equation*}
Consequently, a leading eigenvalue of $\bm{A}(\mathsf{S}_q,\mathsf{S}_q)$ is at least as accurate in approximating the leading eigenvalue magnitude as the Lanczos estimate obtained from this Krylov subspace.
Our experiments show that this additional computation can substantially reduce the number of distinct matrix entries accessed before reaching a given accuracy; see \cref{fig:entries_viewed}.

We observe that classical iterative methods and the submatrix selection method rely on a common assumption: we can construct a sparse unit-norm starting vector that has dimension-independent overlap with a leading eigenvector.
This is impossible in some cases.
For example, when all the entries of $\bm{v}_1$ have magnitude $1/\sqrt{n}$, every unit-norm vector $\bm{x}_0$ satisfies
\begin{equation*}
    |\langle \bm{v}_1, \bm{x}_0\rangle| \leq \sqrt{\operatorname{nnz}(\bm{x}_0)/n}.
\end{equation*}
Thus, any initial vector with a bounded number of nonzeros approaches orthogonality as $n$ grows.
Even in applications where $\bm{v}_1$ is approximately sparse, finding a good $\bm{x}_0$ involves identifying large-magnitude entries of $\bm{v}_1$, which seems hard to do in the worst case.

Nevertheless, a good sparse starting vector is often available in practice.
For example, in the full configuration interaction problem in quantum chemistry \cite{szabo2012modern}, the Hartree--Fock determinant, which we take to be the first basis vector, often has substantial overlap with the leading eigenvector.
The squared overlaps are approximately $0.96$ and $0.74$ for the neon and ozone applications in \cref{sec:experiments_chemistry}.
In such cases, we can simply choose $\bm{x}_0 = \bm{e}_1$.

In other applications, we may know the index $i$ associated with a largest-magnitude diagonal entry.
\Cref{fig:entries_viewed} shows that the simple initialization $\bm{x}_0 = \bm{e}_i$ works surprisingly well, surprisingly often.
For this figure, we analyzed every non-diagonal symmetric matrix in the SuiteSparse collection \cite{KolodziejAznavehBullock:2019,10.1145/2049662.2049663} with dimension $n < 1.4$ million.
We ran the submatrix selection method in \cref{alg:greedy} until achieving relative eigenvalue error $|\hat{\lambda}_1 - \lambda_1|/|\lambda_1| \leq 10^{-5}$ and recorded the percentage of distinct nonzero entries read before convergence.
When an algorithm stalled before reaching this accuracy, we counted the percentage of entries read as $100\%$.
For $21\%$ of the matrices, the method read less than $1\%$ of the nonzero entries, and for $17\%$ it read less than $0.1\%$.
These percentages exclude the cost of finding the initial index $i$.
We also ran Lanczos and the power method from the same initialization.
These methods generally required access to a larger portion of the matrix, sometimes dramatically so, as illustrated by the examples in the right panels.
The substantially reduced access requirements support why submatrix selection methods are so popular in practice.

\begin{figure}[t]
\centering
\includegraphics[width=1\textwidth]{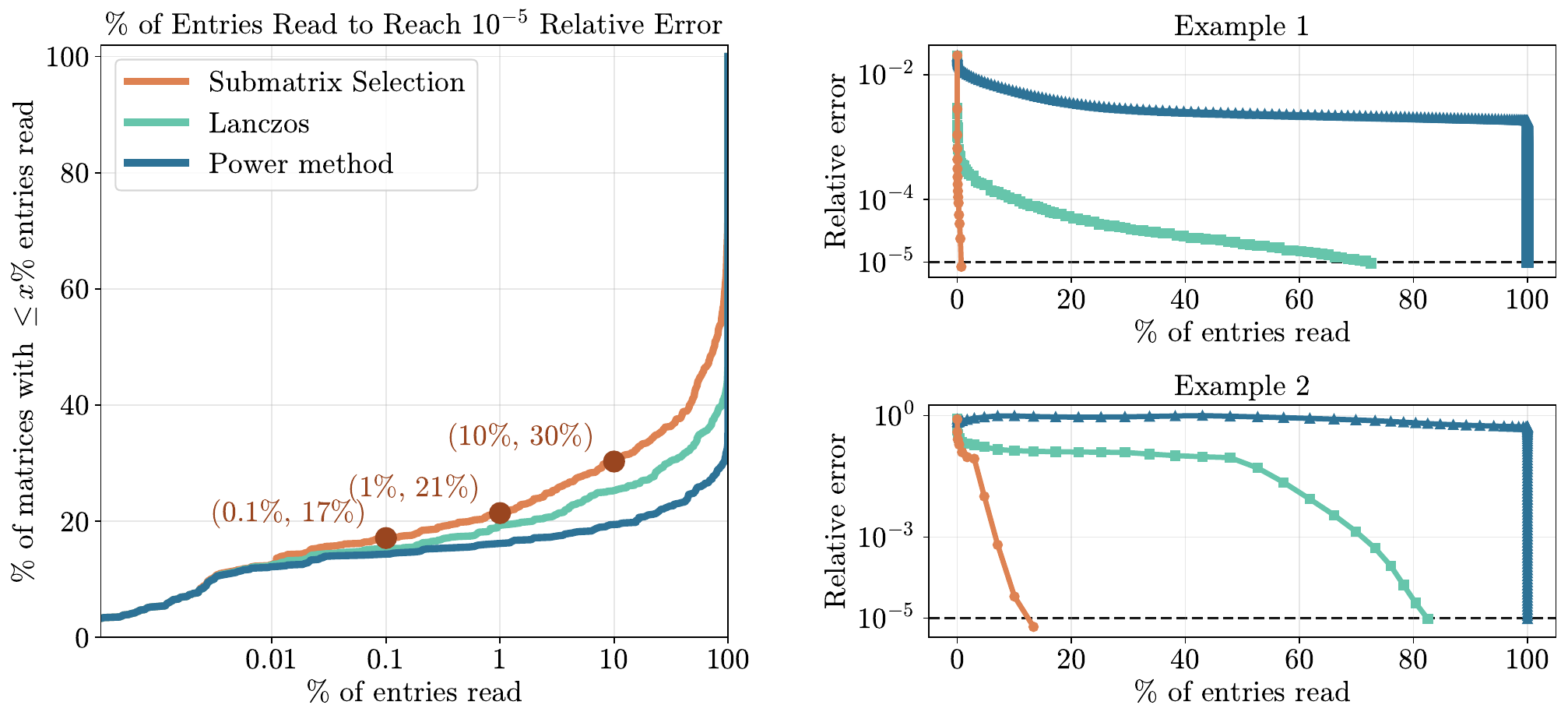}
\caption{ Left panel: Percentage of distinct nonzero entries read before reaching $10^{-5}$ relative eigenvalue error.
A point $(x,y)$ on a curve indicates that the algorithm reads at most $x$ percent of the nonzero entries for $y$ percent of the matrices in SuiteSparse. 
Right panels: Relative error versus percentage of distinct nonzero entries read for two example matrices with SuiteSparse IDs 1412 and 1230.
}
\label{fig:entries_viewed}
\end{figure}

\subsection{Main results for sparsified iterative methods}\label{sec:main_results}

Our preliminary observations suggest that it is often reasonable to assume access to a sparse starting vector that overlaps a leading eigenvector.
However, even with such a starting vector, the classical iterative methods have a common limitation: the number of active coordinates can grow exponentially with the number of iterations, even if it is independent of $n$.
Therefore, we turn to methods that are popular in quantum chemistry, but are less understood rigorously.

State-of-the-art methods control the growth in the number of active coordinates through selective expansion, repeated sparsification, or both.
Selective-expansion methods add only important indices connected to the current support \cite{HuronMalrieuRancurel:1973,HolmesTubmanUmrigar:2016,HernandezVanBeeumenCaprioYang:2021,Wang2019cdfci,Wang2023xcdfci}.
Building on this approach, several submatrix selection methods repeatedly expand the active set, diagonalize the corresponding principal submatrix, and prune the active set to retain only the most significant indices in the approximate top eigenvector obtained \cite{Greer:1995mcci,TubmanLeeTakeshita:2016,SchriberEvangelista:2016}.
Finally, several methods including full configuration interaction quantum Monte Carlo \cite{Boothetal2009fciqmc,Clelandetal2010fciqmc,LuWang:2020} and fast randomized iteration \cite{LimWeare:2017,GreeneWebberWeare:2019, GreeneWebberWeare:2020,GreeneWebberBerkelbach:2022, GreeneWebberSmith:2022} can be interpreted as inexact power methods that randomly expand and contract an active set of indices.
Motivated by these state-of-the-art methods, we ask,
\begin{quote}
\emph{Which dimension-independent convergence guarantees survive repeated sparsification?}
\end{quote}
We first analyze the sparsified power method and obtain a meaningful worst-case convergence guarantee (\cref{sec:sparsified_power}).
We then turn to sparsified submatrix selection and show that it fails dramatically in the worst case (\cref{sec:sparsified_bag_growing}).

\subsubsection{Sparsified power method} \label{sec:sparsified_power}

The classical power method can be sparsified using deterministic truncation of the smallest-magnitude entries \cite{YuanZhang:2013} or random unbiased sparsification \cite{LimWeare:2017}.
Here we focus on the deterministic truncation approach,
\begin{equation}
    \label{eq:deterministic}
    \bm{x}_t = \frac{\bm{\psi}_m(\bm{A} \bm{x}_{t-1})}{\lVert \bm{\psi}_m(\bm{A} \bm{x}_{t-1}) \rVert},
\end{equation}
where $\bm{\psi}_m(\bm{y})$ keeps the $m$ largest-magnitude elements of $\bm{y}$ and sets the rest equal to zero; ties are broken arbitrarily, and $\bm{\psi}_m$ keeps all elements if $m\geq n$.
We establish the following worst-case convergence guarantee, with the proof appearing in \cref{sec:analysis_sparsified_power}.

\begin{theorem}[Sparsified power method convergence]
\label{thm:main_thm}
Suppose $\bm{A} \in \mathbb{C}^{n \times n}$ is a Hermitian matrix with a leading eigenpair $(\lambda_1, \bm{v}_1)$ that has a second-largest-magnitude eigenvalue satisfying $|\lambda_2| < |\lambda_1|$.
Let $\bm{x}_0 \in \mathbb{C}^n$ be a unit-norm starting vector with $\bm{v}_1^* \bm{x}_0 \neq 0$ and $\operatorname{nnz}(\bm{x}_0) \leq m$, and fix a target accuracy $\varepsilon \in (0, |\bm{v}_1^*\bm{x}_0|^2)$.
Let $k$ be the smallest integer satisfying
\begin{equation*}
    \min_{\operatorname{nnz}(\bm{w}) \leq k} \lVert \bm{v}_1 - \bm{w} \rVert^2 \leq \frac{(1 - |\lambda_2 / \lambda_1|)\,\varepsilon}{10}.
\end{equation*}
Run the sparsified power method \cref{eq:deterministic} with
\begin{equation*}
    m \geq \frac{10k}{(1 - |\lambda_2 / \lambda_1|)^2\, |\bm{v}_1^* \bm{x}_0|^2} \quad \text{nonzeros},
\end{equation*}
starting from $\bm{x}_0$.
When the number of iterations is
\begin{equation*}
    q \geq \frac{5}{2(1 - |\lambda_2/\lambda_1|)} \log\left(\frac{2}{\varepsilon |\bm{v}_1^* \bm{x}_0|^2}\right),
\end{equation*}
the method returns an estimated eigenpair $(\hat{\lambda}_1, \hat{\bm{v}}_1) = (\bm{x}_q^* \bm{A} \bm{x}_q, \bm{x}_q)$ that satisfies
\begin{equation}
    \label{eq:the_guarantee}
    |\bm{v}_1^* \hat{\bm{v}}_1|^2 \geq 1 - \varepsilon/2 \quad \text{and therefore} \quad |\lambda_1 - \hat{\lambda}_1| \leq \varepsilon|\lambda_1|.
\end{equation}
\end{theorem}

\Cref{thm:main_thm} establishes dimension-independent convergence for the sparsified power method, while placing a hard cap on the per-iteration runtime.
Let $s$ be an upper bound on the number of nonzero entries per row/column of the target matrix.
Each iteration queries at most $m$ sparse columns and returns at most $s m$ nonzero matrix entries.
The method requires $\mathcal{O}(qsm)$ time to execute $q$ iterations, which is independent of $n$ when $q$, $s$, and $m$ are fixed.

Sparsification preserves the convergence guarantee for the classical power method up to an absolute constant factor.
By \cref{prop:gapped_power}, the classical power method achieves the $\varepsilon$-accuracy guarantee \cref{eq:the_guarantee} when the number of iterations is
\begin{equation*}
    q \geq \frac{1}{2(1 - |\lambda_2/\lambda_1|)}\log\biggl(\frac{2}{\varepsilon\vert \bm{v}_1^\ast \bm{x}_0\vert^2}\biggr).
\end{equation*}
Our bound for the sparsified power method is just five times higher, and the constant five can potentially be reduced through a sharper convergence analysis.

\Cref{thm:main_thm} establishes an eigenvector approximation guarantee $|\bm{v}_1^* \hat{\bm{v}}_1|^2 \geq 1 - \varepsilon/2$, which is equivalent to
\begin{equation*}
    \lVert \bm{v}_1 - (\hat{\bm{v}}_1^* \bm{v}_1)\hat{\bm{v}}_1 \rVert^2 \leq \frac{\varepsilon}{2}.
\end{equation*}
Such a guarantee is only possible when the leading eigenvector $\bm{v}_1$ is well-approximated by an $m$-sparse vector with squared approximation error no larger than $\varepsilon/2$.
However, \cref{thm:main_thm} makes a more stringent requirement: it requires a squared approximation error of order $(1 - |\lambda_2/\lambda_1|) \varepsilon$ using $k$ nonzeros and it requires an oversampling factor $m/k$ of order $(1 - |\lambda_2/\lambda_1|)^{-2} |\bm{v}_1^* \bm{x}_0|^{-2}$.
While constants can certainly be improved, we believe multiplicative factors of these orders are necessary in a worst-case analysis.

To the best of our knowledge, \cref{thm:main_thm} gives the first worst-case convergence guarantee for the deterministically truncated power method applied to a general Hermitian matrix.
Yuan and Zhang \cite{YuanZhang:2013} obtain related recovery guarantees when $\bm{A}$ is a small restricted-norm perturbation of a symmetric matrix with an exactly sparse leading eigenvector.
Earlier papers \cite{ShenHuang:2008,JourneeNesterovRichtarik:2010} studied closely related thresholded power iterations for sparse principal component analysis, but did not obtain approximation guarantees for a leading eigenpair of a general Hermitian matrix.

\subsubsection{Sparsified submatrix selection methods}
\label{sec:sparsified_bag_growing}

Since the power method retains a convergence guarantee under deterministic sparsification, it is natural to ask whether the submatrix selection method does so as well.
One seemingly benign way to sparsify \Cref{alg:greedy} is to restrict the index set $\mathsf{S}_t$ based on the sparsity pattern of the current iterate $\bm{x}_t$.
We set
\begin{equation}
    \label{eq:recursion_1}
    \hat{\mathsf{S}}_t = \{i\ :\ \bm{x}_t(i)\neq 0\}.
\end{equation}
Then, as usual, expand the indices according to
\begin{equation}
    \label{eq:recursion_2}
    \mathsf{S}_{t+1} = \hat{\mathsf{S}}_t \cup \{k : \bm{A}(k,j)\neq 0\text{ for some } j \in \hat{\mathsf{S}}_t\}
\end{equation}
and define the next iterate $\bm{x}_{t+1}$ to be a leading eigenvector of $\mathbf{I}(\cdot, \mathsf{S}_{t+1}) \bm{A}(\mathsf{S}_{t+1}, \mathsf{S}_{t+1}) \mathbf{I}(\mathsf{S}_{t+1}, \cdot)$.
This sparsification rule removes only indices whose coefficients are exactly zero and is therefore substantially less aggressive than practical methods.

The sparsification rule preserves an important stability property.
By the variational principle for Hermitian eigenvalues, each iterate $\bm{x}_{t+1}$ leads to the largest possible value of $|\bm{x}^* \bm{A} \bm{x}|$ over unit-norm vectors $\bm{x}$ supported on $\mathsf{S}_{t+1}$.
The variational principle then implies
\begin{equation*}
    \left| \bm{x}_{t+1}^* \bm{A} \bm{x}_{t+1} \right| \geq \left| \bm{x}_t^* \bm{A} \bm{x}_t \right|,
\end{equation*}
because $\bm{x}_t$ is supported on $\mathsf{S}_{t+1}$.
The method generates a non-decreasing sequence of estimates for the largest eigenvalue magnitude.

Nevertheless, when the initial squared overlap lies below a threshold determined by the relative spectral gap, the method can lose all overlap with the leading eigenvector after a single iteration.
We present the following result, whose proof appears in \cref{sec:analysis_sparsified_submatrix}.

\begin{proposition}[Losing all overlap] \label{prop:greedy_negative_thm}
For any relative eigenvalue gap $r \in (0, 1)$ and for every $\delta$ satisfying
\begin{equation*}
    0 < \delta < \frac{2(1-r)^3}{(1+r)^2(2-r)},
\end{equation*}
there exists a symmetric tridiagonal matrix $\bm{A} \in \mathbb{R}^{4 \times 4}$ and unit-norm initialization $\bm{x}_0 \in \mathbb{R}^4$ with the following properties.
\begin{enumerate}
\item \textbf{Prescribed eigenvalue ratio.} The largest-magnitude eigenvalues are
    \begin{equation*}
    \lambda_1 = 1, \qquad \lambda_2 = 1 - r.
\end{equation*}
\item \textbf{Sparse leading eigenvector.} The leading eigenvector $\bm{v}_1$ has support $\{1, 2, 3\}$.
\item \textbf{Prescribed initial overlap.} The initial iterate has nontrivial overlap with the target eigenvector,
    \begin{equation*}
    |\bm{v}_1^* \bm{x}_0|^2 = \delta.
\end{equation*}
\item \textbf{Zero overlap of subsequent iterates.} When we apply the sparsified submatrix method \eqref{eq:recursion_1}--\eqref{eq:recursion_2}, it produces iterates that have zero overlap with the target eigenvector,
    \begin{equation*}
    \bm{v}_1^* \bm{x}_t = 0, \qquad t = 1, 2, \ldots
\end{equation*}
\end{enumerate}
\end{proposition}

Without loss of generality, \cref{prop:greedy_negative_thm} assumes the leading eigenvalue is $\lambda_1 = 1$, and all other eigenvalues are in the range $[-(1-r), 1-r]$.
The proposition then identifies a threshold value for the initial squared overlap,
\begin{equation*}
    |\bm{v}_1^* \bm{x}_0|^2 = \frac{2(1-r)^3}{(1+r)^2(2-r)}.
\end{equation*}
For example, this threshold exceeds $0.95$ when the spectral gap is $r < 0.01$, as occurs in the chemistry experiments in \cref{sec:experiments_chemistry}.
Asymptotically, it is $1 - \frac{9}{2}r + \mathcal{O}(r^2)$ as $r \rightarrow 0$.

When the initial squared overlap lies below the threshold, the iterates can immediately lose all overlap with the leading eigenvector.
Thus, sparsified submatrix selection lacks the regular convergence guarantee of the sparsified power method, which provides a high-quality leading eigenvector estimate given any nonzero initial overlap, provided the leading eigenvector is compressible and the number of nonzeros is chosen suitably high.

Conversely, when the squared overlap starts above the threshold, complete loss of overlap is impossible, because the Rayleigh quotient exceeds $1-r$ in the very first iteration.
Specifically, we can evaluate the shifted Rayleigh quotient
\begin{equation*}
    \frac{\bm{z}^* \bm{A} \bm{z}}{\bm{z}^* \bm{z}} - (1-r) = \frac{\bm{z}^* (\bm{A} - (1-r) \mathbf{I}) \bm{z}}{\bm{z}^* \bm{z}}
\end{equation*}
with the vector $\bm{z} = (2 \bm{A} + (1-r) \mathbf{I}) \bm{x}_0$.
Since $\bm{x}_0$ has unit norm and $\bm{A} + (1-r)\mathbf{I}$ is positive-semidefinite,
\begin{align*}
    \bm{z}^* (\bm{A} - (1-r) \mathbf{I}) \bm{z} &= \bm{x}_0^* \bigl(2\bm{A} - (1-r) \mathbf{I} \bigr)^2 \bigl(\bm{A} + (1-r) \mathbf{I} \bigr) \bm{x}_0 - 2(1-r)^3 \\
    &\geq |\bm{v}_1^* \bm{x}_0|^2 (1+r)^2(2-r) - 2(1-r)^3.
\end{align*}
This last quantity is strictly positive when $|\bm{v}_1^* \bm{x}_0|^2$ exceeds the threshold $2(1-r)^3 / [(1+r)^2 (2-r)]$.
The variational principle then guarantees
\begin{equation*}
    \bm{x}_1^* \bm{A} \bm{x}_1 > 1-r,
\end{equation*}
because the approximation space for the leading eigenvector includes $\bm{z}$.
Monotonicity then guarantees
\begin{equation*}
    \bm{x}_t^* \bm{A} \bm{x}_t > 1-r, \qquad t=2,3,\ldots,
\end{equation*}
which forces every iterate to have nonzero overlap with the leading eigenvector.

In summary, the counterexample isolates a specific failure mechanism: when the first Ritz value has magnitude at most $|\lambda_2|$, sparsification can destroy all overlap with the leading eigenvector.
Once the first Ritz value has magnitude exceeding $|\lambda_2|$, monotonicity rules out this catastrophic failure.
Yet despite the worst-case theory placing heavy restrictions on the initial overlap, sparsified submatrix-selection methods reach a given accuracy in fewer outer iterations than the truncated power method on our chemical systems, as shown in \cref{sec:experiments_chemistry}.
Identifying structural conditions that explain this success remains an open problem.

\subsection{Related work}

Several alternative approaches are designed to approximate leading eigenpairs of a matrix without reading all the nonzero entries.

One group of algorithms estimates outlying eigenvalues or eigenvectors from a uniformly sampled principal submatrix \cite{BhattacharjeeDexterDrineas:2023,SwartworthWoodruff:2025} or uniformly sampled set of columns \cite{BhattacharjeeMuscoRutkowski:2026}.
The uniform-sampling eigenvalue bounds scale with $\max_{i,j} |\bm{A}(i,j)|\,n$, so they give dimension-independent relative-error guarantees only for matrices with $\Omega(n^2)$ nonzeros.
Moreover, a dimension-independent number of uniform samples can completely miss the support of a sparse eigenvector.
Some papers obtain stronger bounds using sparsity- or norm-proportional sampling, but those results require sampling access not available in our sparse-column model.

Other algorithms compute low-rank approximations for targeted matrix classes, such as positive-semidefinite matrices \cite{MuscoWoodruff:2017,ChenEpperlyTropp:2025}, kernel matrices \cite{MuscoMusco:2017,BackursIndykMusco:2021,BakshiIndykKacham:2023}, or positive-semidefinite Toeplitz or Hankel matrices \cite{MuscoSheth:2024,KapralovMuscoSheth:2026}.
The algorithms for general positive-semidefinite matrices retain at least linear dependence on $n$, while the Toeplitz and Hankel algorithms achieve polylogarithmic dependence on $n$ by exploiting much stronger matrix structure.

Last, related algorithms approximate the eigenvalue density for graph adjacency and Laplacian matrices \cite{Cohen-SteinerKongSohler:2018,BravermanKrishnanMusco:2022,JinKarmarkarMusco:2024}.
Without further assumptions, the eigenvalue density is not directly comparable to the outlying eigenvalues, since a fixed number of outlying eigenvalues accounts for only a $\mathcal{O}(1/n)$ fraction of the distributional mass.

\section{Quantum chemistry experiments}
\label{sec:experiments_chemistry}

This section presents case studies of large-scale eigenproblems in quantum chemistry.
\Cref{sec:fci_matrices} describes how massive sparse matrices arise naturally in this context, and \cref{sec:neon,sec:ozone} report numerical results for neon and ozone.

\subsection{Large-scale eigenproblems in electronic structure theory}
\label{sec:fci_matrices}

Consider an electronic-structure problem with $N$ electrons.
The electronic Hamiltonian acts on wavefunctions of $3N$ continuous spatial coordinates and $N$ discrete spin variables.
We form a Galerkin discretization by forming Slater determinants from $M$ spatial orbitals, corresponding to $L=2M$ spin-orbitals.
Each Slater determinant occupies $N$ of the $L$ spin-orbitals, so the full determinant basis has dimension
\begin{equation*}
    n=\binom{L}{N},
\end{equation*}
before imposing additional spin or spatial-symmetry restrictions.

Although $\bm{H}$ can be exponentially large, the number of nonzero entries per row or column is polynomial in the numbers of electrons and orbitals.
The electronic Hamiltonian is a sum of one- and two-electron operators, and two determinants have a zero Hamiltonian matrix element unless they differ by at most a double excitation \cite{szabo2012modern}.
Consequently, the maximum column sparsity $s$ satisfies
\begin{equation*}
    s=\max_j\operatorname{nnz}\bigl(\bm{H}(\mathord{\cdot},j)\bigr) \leq \sum_{\ell=0}^2 \binom{N}{\ell}\binom{L-N}{\ell}.
\end{equation*}
Fixing the numbers of spin-up and spin-down electrons and restricting to a point-group symmetry block give sharper bounds.
Each nonzero matrix element can be assembled on demand as a linear combination of one- and two-electron orbital integrals.

Many electronic-structure algorithms seek the ground state, defined after discretization as the eigenvector of $\bm{H}$ with the smallest real eigenvalue.
To promote this goal, orbitals are constructed that encourage empirical compressibility of the ground state.
The canonical Hartree--Fock orbitals are eigenvectors of a self-consistent Fock operator, ordered by increasing orbital energy.
We order the determinant basis so that its first element is the Hartree--Fock determinant.
For many systems near their equilibrium geometries, the Hartree--Fock determinant provides an accurate reference, and the remaining correlation is compactly represented by excitations from that determinant.

\begin{algorithm}[t]
\caption{Sparsified power method}
\label{alg:sparsified}
\raggedright
{\bfseries Input}: Hermitian matrix $\bm{A}\in\mathbb{C}^{n\times n}$; truncation parameter $m$; unit vector $\bm{x}_0$ satisfying $\operatorname{nnz}(\bm{x}_0)\leq m$; number of iterations $q$.
\par {\bfseries Output}: Approximation $(\hat{\lambda}_1,\hat{\bm{v}}_1)$ to a leading eigenpair of $\bm{A}$.
\par
\begin{algorithmic}[1]
\FOR{$t=1, \ldots, q$}
\STATE $\bm{y}_t\gets \bm{A} \bm{x}_{t-1}/\lVert \bm{A} \bm{x}_{t-1}\rVert$
\STATE $\bm{z}_t \gets \bm{\psi}_m(\bm{y}_t)$
\COMMENT{Keep $m$ largest-magnitude elements and zero out the rest}
\STATE $\bm{x}_t \gets \bm{z}_t / \lVert \bm{z}_t \rVert$
\ENDFOR
\STATE {\bfseries Return} $(\hat{\lambda}_1, \hat{\bm{v}}_1) = (\bm{x}_q^\ast \bm{A} \bm{x}_q, \bm{x}_q)$.
\end{algorithmic}
\end{algorithm}

\begin{algorithm}[t]
\caption{Sparsified submatrix selection method}
\label{alg:sparsified_greedy}
\raggedright
{\bfseries Input}: Hermitian matrix $\bm{A}\in\mathbb{C}^{n\times n}$; truncation parameter $m$; unit vector $\bm{x}_0$ satisfying $\operatorname{nnz}(\bm{x}_0)\leq m$; number of iterations $q$.\par {\bfseries Output}: Approximation $(\hat{\lambda}_1,\hat{\bm{v}}_1)$ to a leading eigenpair of $\bm{A}$.
\par
\begin{algorithmic}[1]
\FOR{$t=1, \ldots, q$}
\STATE $\bm{y}_t\gets \bm{A} \bm{x}_{t-1}/\lVert \bm{A} \bm{x}_{t-1}\rVert$
\STATE $\bm{z}_t \gets \bm{\psi}_m(\bm{y}_t)$
\COMMENT{Keep $m$ largest-magnitude elements and zero out the rest}
\STATE $\mathsf{S}_t\gets \operatorname{supp}(\bm{z}_t)$
\STATE $(\lambda_t , \tilde{\bm{x}}_t) \gets$ a leading eigenpair of $\bm{A}(\mathsf{S}_t,\mathsf{S}_t)$
\STATE $\bm{x}_t\gets \mathbf{I}(\mathord{\cdot}, \mathsf{S}_t) \tilde{\bm{x}}_t$
\ENDFOR
\STATE {\bfseries Return} $(\hat{\lambda}_1, \hat{\bm{v}}_1) = (\lambda_q, \bm{x}_q)$.
\end{algorithmic}
\end{algorithm}

We initialize our methods with $\bm{x}_0=\bm{e}_1$.
We estimate the ground-state eigenpair using the truncated power method (\cref{alg:sparsified}) and the sparsified submatrix selection method (\cref{alg:sparsified_greedy}).
TPM deterministically sparsifies each normalized matrix-vector product, whereas submatrix selection uses the retained support to define a principal submatrix and returns its leading eigenvector.
We retain only the largest entries of the matrix-vector product, making \cref{alg:sparsified_greedy} more aggressive than the support-expansion rule analyzed in \cref{sec:sparsified_bag_growing}.
Both methods maintain iterates satisfying $\operatorname{nnz}(\bm{x}_t)\leq m$.

For the neon and ozone test systems, the ground-state eigenvalue is also the unique largest-magnitude eigenvalue, so we apply our methods directly to $\bm{H}$.
We define $\bm{H}$ to include the frozen-core energy constant and add the nuclear-repulsion energy to the reported energies.
For our neon experiments, we declare convergence at the first iteration satisfying
\begin{equation*}
    |\hat{E}_t-\hat{E}_{t-1}| < 10^{-11}\,\mathrm{Ha},
\end{equation*}
where $\mathrm{Ha}$ denotes the hartree, the atomic energy unit used for the eigenvalues.
For our ozone experiments, neither method reaches the convergence threshold.
We compute orbitals, symmetry labels, and integrals with PySCF \cite{SunBerkelbachBlunt:2018,SunZhangBanerjee:2020}.

More generally, suppose that the eigenvalues of $\bm{H}$ satisfy
\begin{equation*}
    E_0 < E_1 \leq \cdots \leq E_{n-1},
\end{equation*}
where $E_0$ is the ground-state energy.
Applying the power method to
\begin{equation*}
    \bm{A}_\mu=\bm{H}-\mu\mathbf{I}
\end{equation*}
gives the relative spectral gap
\begin{equation*}
    1 - \frac{\max_{1 \le i \le n-1}|E_i-\mu|}{|E_0-\mu|}.
\end{equation*}
The shift that maximizes this quantity is
\begin{equation*}
    \mu_\star = \frac{E_1+E_{n-1}}{2}.
\end{equation*}
The eigenvalues $E_1$ and $E_{n-1}$ are generally unknown, so in practice one must estimate the shift or compare several candidates.
A poor shift can substantially slow TPM convergence, whereas the sparsified submatrix selection method is empirically much less sensitive to the shift.

\subsection{Neon} \label{sec:neon}

For our experiments with the neon atom, we use the aug-cc-pVDZ basis \cite{10.1063/1.462569}.
To reduce the problem size, we freeze the $1s^2$ core, leaving $N=8$ active electrons in $L=44$ active spin-orbitals.
We then restrict to the ground-state point-group symmetry block within $M_S=0$, represented in the Slater determinant basis.
The resulting block contains $n\approx 7\times10^6$ basis elements and has column sparsity $s\leq 2\times10^3$.

\begin{figure}[t]
\centering
\includegraphics[width=\linewidth]{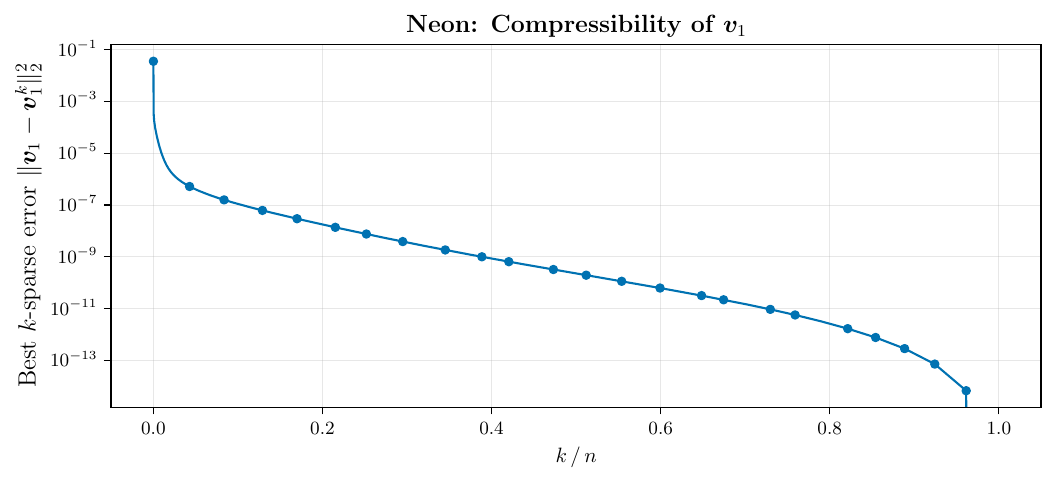}
\caption{Error $\lVert \bm{v}_1-\bm{v}^{(k)}_1\rVert_2^2$, where $\bm{v}^{(k)}_1$ keeps the $k$ largest-magnitude elements of $\bm{v}_1$ and zeros out the rest.}
\label{fig:neon_compressibility}
\end{figure}

The neon symmetry block is small enough that we can use ARPACK \cite{LehoucqSorensenYang:1998} to compute reference eigenpairs.
We find that the relative eigenvalue gap is $r=1-|\lambda_2/\lambda_1|\approx 0.0057$.
Additionally, the leading eigenvector $\bm{v}_1$ is empirically compressible, as illustrated in \cref{fig:neon_compressibility}, and its squared overlap with the Hartree--Fock determinant is
\begin{equation*}
    |\bm{v}_1^* \bm{x}_0|^2 \approx 0.96.
\end{equation*}
This high initial overlap nevertheless lies slightly below the threshold appearing in \cref{sec:sparsified_bag_growing}, which evaluates to
\begin{equation*}
    \frac{2(1-r)^3}{(1+r)^2(2-r)} \approx 0.975.
\end{equation*}
With this initial overlap and spectral gap, complete loss of overlap after a single iteration is possible in the worst case.

The worst case does not happen for neon, however.
At the truncation fraction $m/n=0.021$, two iterations of sparsified submatrix selection reduce the relative eigenvalue error below the chemical-accuracy threshold.
Here chemical accuracy means an absolute energy error below $1\,\mathrm{kcal\,mol^{-1}}$ \cite{ChienHolmes2018}, which corresponds to the relative-error threshold $1.2\times10^{-5}$ shown in the figure.

\begin{figure}[t]
\centering
\includegraphics[width=\textwidth]{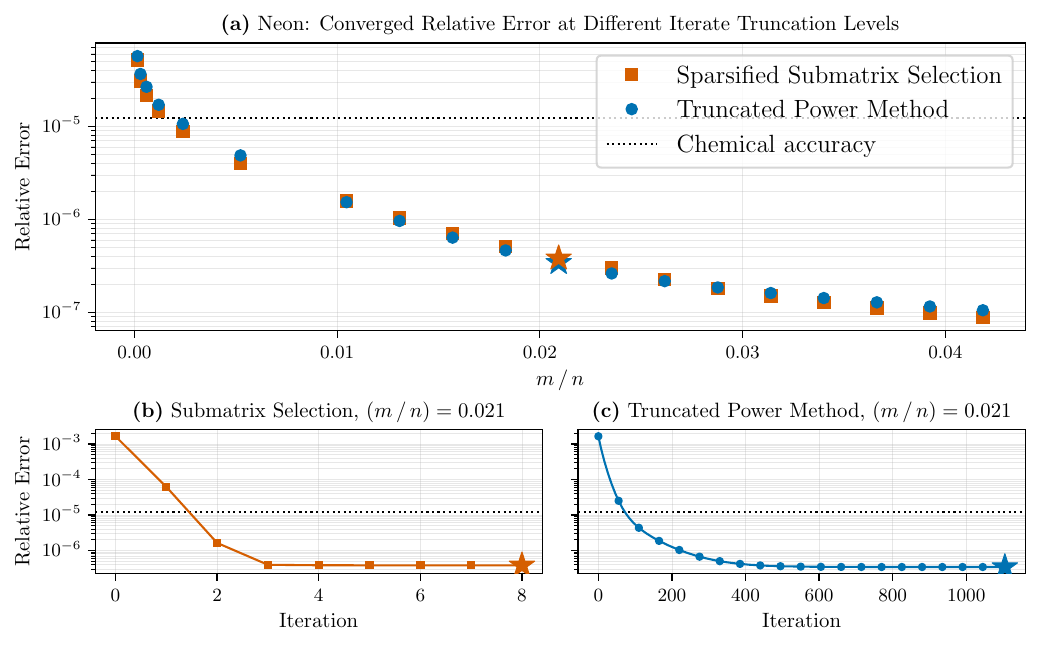}
\caption{Convergence of the truncated power method (TPM) and the sparsified submatrix selection method for the neon Hamiltonian, restricted to the ground-state symmetry block of dimension $n\approx6.69\times10^6$.
Relative eigenvalue error is $|\hat{E}-E_{\mathrm{ref}}|/|E_{\mathrm{ref}}|$, where $E_{\mathrm{ref}}$ is the benchmark ground-state energy computed from the full, untruncated symmetry block.
Panel~(a) shows terminal relative error versus the terminal support fraction $\operatorname{nnz}(\hat{\bm{v}}_1)/n$ as the truncation fraction is varied.
Panels~(b) and~(c) show the error by iteration for sparsified submatrix selection and TPM, respectively, at $m/n=0.021$.
Stars mark the terminal values of these runs and the corresponding points in panel~(a).
The dotted line is the relative-error threshold corresponding to an absolute error of $1\,\mathrm{kcal\,mol^{-1}}$.}
\label{fig:neon_results}
\end{figure}

TPM and sparsified submatrix selection attain similar terminal errors for all the truncation parameters we tried.
However, they converge at sharply different rates.
Sparsified submatrix selection converges in only a few iterations, whereas TPM requires hundreds.
\Cref{fig:neon_results} reports the unshifted runs.
In an additional run, we applied the optimal shift $\bm{A}_{\mu_\star}=\bm{H}-\mu_{\star} \mathbf{I}$ with $\mu_{\star}\approx -102\,\mathrm{Ha}$, which reduced the TPM's iteration count required for convergence by a factor of approximately five.
Nevertheless, TPM still required over an order of magnitude more iterations than sparsified submatrix selection.

\subsection{Ozone} \label{sec:ozone}

For our experiments with the ozone molecule ($\text{O}_3$), we use the cc-pVTZ basis \cite{10.1063/1.456153} at the equilibrium geometry $\theta=116.6^\circ$ and $r_{\mathrm{O-O}}=1.2885\,\text{\AA}$ \cite{ChienHolmes2018}.
To reduce the problem size, we freeze the three oxygen $1s$ orbitals, leaving $N=18$ active electrons in $L=174$ active spin-orbitals.
We then restrict to the $A_1$ spatial-symmetry block of $C_{2v}$ and, within $M_S=0$, to the even spin-parity sector, represented in a spin-parity-adapted basis rather than the Slater determinant basis \cite{HolmesUmrigarSharma:2017}.
The Hartree--Fock basis element remains unchanged, while every other basis element is a combination of at most two Slater determinants.
The resulting block contains $n\approx 3\times10^{22}$ basis elements and has column sparsity $s\leq 3\times10^5$.

Ozone is a more challenging test problem than neon because the squared overlap of the reference ground state with the Hartree--Fock determinant is approximately $0.74$, compared with $0.96$ for neon.
With a truncation parameter $m=10^6$, \cref{fig:ozone} shows that the sparsified submatrix selection method reduces the ozone error substantially faster than TPM across the plotted outer iterations.
However, neither method reaches chemical accuracy or satisfies our numerical convergence criterion during the reported run.
We therefore turn to a method that separates the truncation budget from the variational-space budget and adds a perturbative correction.

\begin{figure}[t]
\centering
\includegraphics[width=\textwidth]{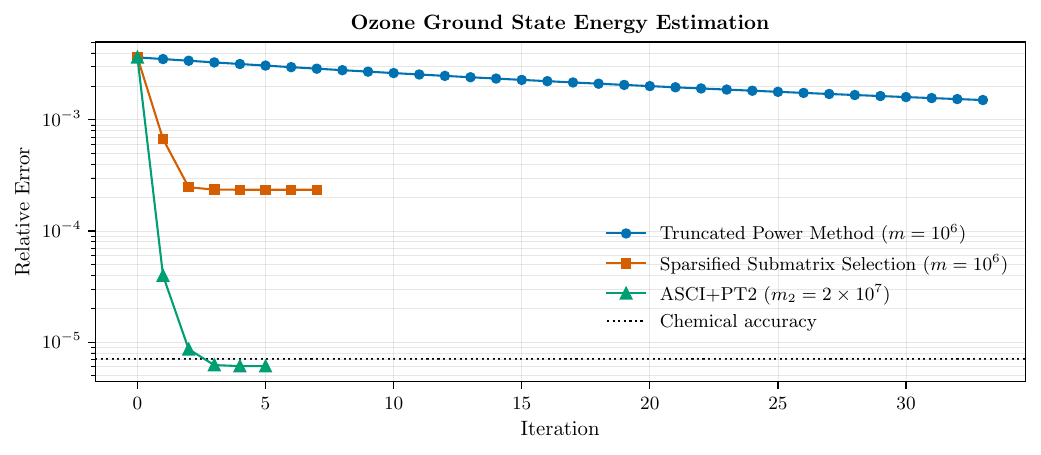}
\caption{Relative energy error by outer iteration for TPM, sparsified submatrix selection, and ASCI+PT2 on the ozone Hamiltonian, whose dimension is $n\approx3\times10^{22}$.
Relative error is measured against the extrapolated reference energy $E_{\mathrm{ref}}=-225.13956\,\mathrm{Ha}$; the dotted line marks chemical accuracy.}
\label{fig:ozone}
\end{figure}

Adaptive Sampling Configuration Interaction (ASCI) \cite{TubmanLeeTakeshita:2016} is a submatrix-selection method that controls the number of Hamiltonian elements accessed during the search step with a truncation parameter $m_1$.
At each iteration, ASCI retains the $m_1$ largest-magnitude entries of the previous variational eigenvector $\bm{x}_{t-1}$ to form $\hat{\bm{x}}_{t-1}$.
It ranks each candidate index $i$ according to
\begin{equation*}
    r_t(i) = \begin{cases}
    \hat{\bm{x}}_{t-1}(i), & i\in\operatorname{supp}(\hat{\bm{x}}_{t-1}),\\[3pt]
    \frac{(\bm{H}\hat{\bm{x}}_{t-1})(i)}{\lambda_{t-1}-\bm{H}(i,i)}, & \text{otherwise},
    \end{cases}
\end{equation*}
where $\lambda_{t-1}$ is the previous variational energy.
It then selects up to $m_2$ indices with the largest values of $|r_t(i)|$ to form the active set $\mathsf{S}_t$ and constructs the next iterate from the ground-state eigenpair of $\bm{H}(\mathsf{S}_t,\mathsf{S}_t)$.
By choosing a large variational-space budget $m_2 > m_1$, it allows exploration of an enhanced variational subspace.

We make another modification based on adding a perturbative correction.
For each active set $\mathsf{S}_t$, we augment the variational energy $\lambda_t$ with the screened Epstein--Nesbet correction
\begin{equation*}
    \Delta E_2^{(\tau)} = \sum_{i\notin\mathsf{S}_t} \frac{\left|\sum_{j\in\mathsf{S}_t, \, |\bm{H}(i,j)\bm{x}_t(j)|>\tau} \bm{H}(i,j)\bm{x}_t(j)\right|^2}{\lambda_t-\bm{H}(i,i)}.
\end{equation*}
Here $\tau\geq0$ is a screening threshold measured in hartrees; $\tau=0$ gives the unscreened Epstein--Nesbet correction.
The resulting ASCI+PT2 estimate of the electronic energy is $\lambda_t+\Delta E_2^{(\tau)}$.
With $(m_1,m_2,\tau) = (10^5,\allowbreak\;2\times10^7,\allowbreak\;10^{-4}\,\mathrm{Ha})$, \cref{fig:ozone} shows that ASCI+PT2 reaches chemical accuracy relative to our extrapolated reference after only a few outer iterations.

Last, we construct an ozone reference energy by extrapolating five ASCI+PT2 calculations toward zero perturbative correction \cite{HolmesUmrigarSharma:2017,10.1021/acs.jctc.1c00272}.
We use the following parameter sets $(m_1,m_2,\tau)$ for the reference-energy extrapolation:
\begin{equation*}
    (m_1,m_2,\tau)\in \left\{ \begin{gathered}
    (10^4,\,2\times10^5,\,0),\quad (2.5\times10^4,\,5\times10^5,\,0),\\
    (5\times10^4,\,10^6,\,2\times10^{-4}),\quad (10^5,\,2\times10^6,\,2\times10^{-4}),\\
    (1.5\times10^5,\,3\times10^6,\,2\times10^{-4})
    \end{gathered} \right\},
\end{equation*}
where $\tau$ is measured in hartrees.
We fit the corrected total energies linearly against the corresponding values of $|\Delta E_2^{(\tau)}|$ and take the intercept at $|\Delta E_2^{(\tau)}|=0$.
The fit in \cref{fig:extrapolation} has $R^2=0.9966$ and gives $E_{\mathrm{ref}}=-225.13956\,\mathrm{Ha}$.

\begin{figure}[t]
\centering
\includegraphics[width=\textwidth]{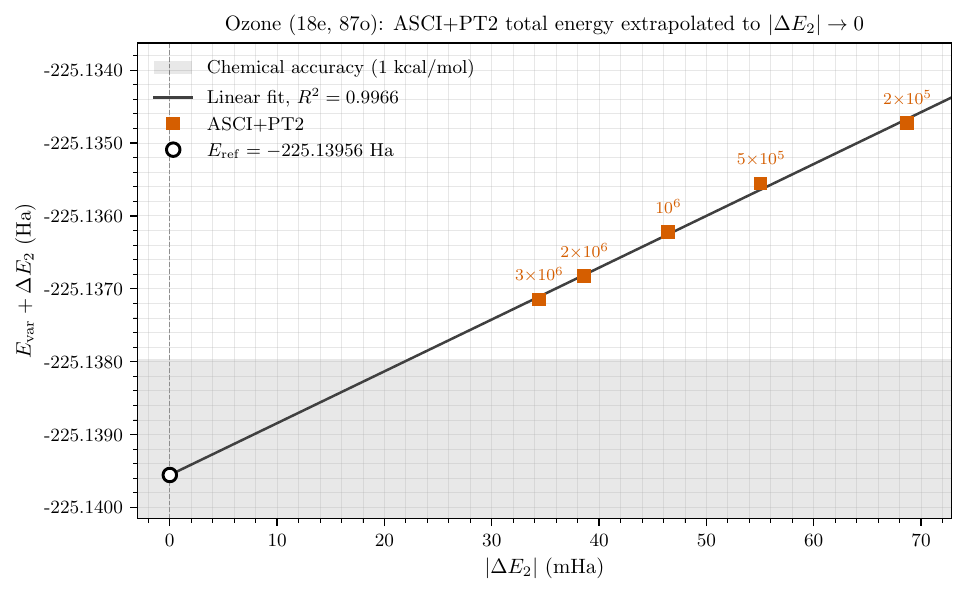}
\caption{Linear extrapolation of the ozone ASCI+PT2 energy estimates to zero perturbative correction.
The fit has $R^2=0.9966$ and intercept $E_{\mathrm{ref}}=-225.13956\,\mathrm{Ha}$.}
\label{fig:extrapolation}
\end{figure}

These chemistry experiments illustrate both the usefulness and the practical limitations of the truncated power method.
They also show how successful practical methods rely on ingredients that go far beyond our worst-case convergence theory.
Submatrix selection is often combined with sparsification and perturbative correction to supply an accurate reference at an acceptable cost.

\section{Analysis of classical iterative algorithms}
\label{sec:analysis_classical}

As discussed in \Cref{sec:introduction}, standard iterative methods obtain dimension-independent runtime for sparse matrices, given a sparse initial vector that has dimension-independent overlap with a leading eigenvector.
In \cref{sec:lanczos}, we formalize this claim through an analysis of the classic Lanczos method.
In \cref{sec:power}, we show that a slightly weaker bound can be obtained via a simpler analysis of the power method.

\subsection{Analysis of Lanczos method} \label{sec:lanczos}

\begin{algorithm}[t]
\caption[Lanczos method]{Lanczos method}
\label{alg:lanczos}
\raggedright
{\bfseries Input}: Hermitian matrix $\bm{A}\in \mathbb{C}^{n \times n}$.
Starting unit-norm vector $\bm{x}_0 \in \mathbb{C}^{n}$.
Krylov depth $q$.
\par {\bfseries Output}: Approximation $(\hat{\lambda}_1, \hat{\bm{v}}_1)$ to largest-magnitude eigenpair $(\lambda_1,\bm{v}_1)$ of $\bm{A}$.\par
\begin{algorithmic}[1]
\STATE Using the Lanczos recurrence, compute an orthonormal basis $\bm{Q}$ for $\operatorname{span}\{\bm{x}_0,\bm{A}\bm{x}_0,\ldots,\bm{A}^q\bm{x}_0\}$, while forming $\bm{T}=\bm{Q}^*\bm{A}\bm{Q}$.
If the Krylov space becomes invariant earlier, stop and use the smaller attained basis.
\STATE Compute a largest-magnitude eigenpair $(\tilde{\lambda}_1, \tilde{\bm{v}}_1)$ of $\bm{T}$.
\STATE {\bfseries Return}: $(\hat{\lambda}_1, \hat{\bm{v}}_1) = (\tilde{\lambda}_1, \bm{Q} \tilde{\bm{v}}_1)$.
\end{algorithmic}
\end{algorithm}

Here we analyze the standard Lanczos algorithm \cite{Lanczos:1950} presented in \cref{alg:lanczos}.
We establish a gap-independent convergence bound in terms of the initial overlap with a leading eigenvector.
In the classical analysis of the Lanczos method \cite{KuczynskiWozniakowski:1992,Tropp:2022}, a uniformly random starting direction leads to relative eigenvalue error $\varepsilon$ after $\mathcal{O}(\varepsilon^{-1/2}\log n)$ iterations, either in expectation or at a fixed failure probability.
Here we prove a new complexity bound of $\mathcal{O}(\varepsilon^{-1/2})$ iterations, assuming dimension-independent overlap with the leading eigenvector.
If each column of $\bm{A}$ has at most $s$ nonzeros, then forming the projected problem on $\operatorname{span}\{\bm{x}_0,\bm{A}\bm{x}_0,\ldots, \bm{A}^q\bm{x}_0\}$ requires reading at most $\operatorname{nnz}(\bm{x}_0)(s+1)^{q+1}$ matrix entries, yielding \cref{prop:informal_introduction}.

We prove \cref{thm:gap_indep} using variational and polynomial arguments \cite{Saad:2011,Tropp:2022}.
We express the polynomial filter in terms of Chebyshev polynomials of the fourth kind \cite{MasonHandscomb:2003}, which makes the required estimates particularly simple.

\begin{theorem}[Lanczos convergence in terms of starting vector]
\label{thm:gap_indep}
Suppose $\bm{A} \in \mathbb{C}^{n \times n}$ is a Hermitian matrix with a leading eigenpair $(\lambda_1, \bm{v}_1)$, and let $\bm{x}_0 \in \mathbb{C}^n$ be a unit-norm starting vector with $\bm{v}_1^* \bm{x}_0 \neq 0$.
Fix a target accuracy $\varepsilon \in (0, 1)$.
For any Krylov depth satisfying
    \begin{equation*}
    q \geq \frac{1}{\sqrt{8 \varepsilon}} \log\Biggl(\frac{4}{|\bm{v}_1^* \bm{x}_0|^2}\Biggr),
\end{equation*}
\cref{alg:lanczos} returns an estimate $\hat{\lambda}_1$ satisfying $(1 - \varepsilon) |\lambda_1| \leq |\hat{\lambda}_1| \leq |\lambda_1|$.
\end{theorem}
\begin{proof}
The method is exact for $\lambda_1 = 0$, so we focus on the case $\lambda_1 \neq 0$.
Scaling $\bm{A}$ by $1/\lambda_1$ does not change the relative eigenvalue error, so we assume $\lambda_1=1$.
Take an eigendecomposition $\bm{A} = \sum_{i=1}^n \lambda_i \bm{v}_i \bm{v}_i^*$ and introduce coefficients $c_i = |\bm{v}_i^* \bm{x}_0|^2$ for $i = 1, \ldots, n$.
Let $\hat{\lambda}_1$ be the eigenvalue estimate returned at Krylov depth $q$.
Then by the variational argument in the proof of \cite[Thm.~6.4]{Saad:2011},
    \begin{equation*}
    0 \leq 1 - |\hat{\lambda}_1| \leq \frac{\sum_{i=2}^n c_i (1 - \lambda_i) p(\lambda_i)^2}{\sum_{i=1}^n c_i p(\lambda_i)^2},
\end{equation*}
where $p$ is any real polynomial of degree at most $q$ for which $p(\bm{A}) \bm{x}_0 \neq 0$.
The major steps are (i) choosing a polynomial, (ii) bounding $p(1)$ from below, and (iii) completing the argument.

\textbf{(i) Choosing a polynomial.} We set
    \begin{equation*}
    p(x) = W_q\biggl(\frac{2 x + \varepsilon}{2 - \varepsilon}\biggr),
\end{equation*}
where $W_q$ is the $q$th Chebyshev polynomial of the fourth kind.
For $x = \cos \theta$ with $\theta \in (0, \pi]$, it is defined by
    \begin{equation*}
    W_q(x) = W_q(\cos \theta) = \frac{\sin\bigl((q + 1/2) \theta\bigr)}{\sin\bigl(\theta/2\bigr)}
\end{equation*}
with its value at $\theta=0$ defined by continuity, so
    \begin{equation*}
    (1 - x) W_q(x)^2 = (1 - \cos \theta) \frac{\sin^2\bigl((q + 1/2) \theta\bigr)}{\sin^2\bigl(\theta/2\bigr)} = 2 \sin^2\bigl(\theta/2\bigr) \frac{\sin^2\bigl((q + 1/2) \theta\bigr)}{\sin^2\bigl(\theta/2\bigr)} \leq 2.
\end{equation*}
We have chosen the Chebyshev polynomial of the fourth kind because of this property and because it increases quickly for $x > 1$.

\textbf{(ii) Bounding $p(1)$ from below.} Now we make a key estimate.
For $x = (\alpha + \alpha^{-1})/2$ with $\alpha > 1$, the Chebyshev polynomial satisfies
    \begin{equation*}
    W_q(x) = W_q\left(\frac{\alpha + \alpha^{-1}}{2}\right) = \frac{\alpha^{q + 1/2} - \alpha^{-(q + 1/2)}}{\alpha^{1/2} - \alpha^{-1/2}}.
\end{equation*}
We choose the parameter
    \begin{equation*}
    \alpha = \frac{1 + \sqrt{\varepsilon/2}}{1 - \sqrt{\varepsilon/2}}
\end{equation*}
and make the calculation
\begin{align*}
    \frac{\varepsilon}{2 - \varepsilon} p(1)^2 &= \frac{\varepsilon}{2 - \varepsilon} W_q\left(\frac{2 + \varepsilon}{2 - \varepsilon}\right)^2 \\
    &= \frac{\alpha + \alpha^{-1} - 2}{4} W_q\left(\frac{\alpha + \alpha^{-1}}{2}\right)^2 \\
    &= \frac{1}{4} \bigl(\alpha^{q + 1/2} - \alpha^{-(q + 1/2)}\bigr)^2.
\end{align*}
By assumption, the depth parameter $q$ satisfies
    \begin{equation*}
    q \geq \frac{\log(4/c_1)}{\sqrt{8\varepsilon}}.
\end{equation*}
Using the inequality $2s \leq \log((1 + s)/(1-s))$, it follows
    \begin{equation*}
    \log(4/c_1) \leq 2q \sqrt{2\varepsilon} \leq 2q\log\left(\frac{1 + \sqrt{\varepsilon/2}}{1 - \sqrt{\varepsilon/2}}\right) = 2q \log \alpha.
\end{equation*}
Consequently,
    \begin{equation*}
    \alpha^{q + 1/2} \geq \frac{2}{\sqrt{c_1}} \quad \text{and} \quad \alpha^{-q-1/2} \leq \frac{\sqrt{c_1}}{2}.
\end{equation*}
We conclude that
    \begin{equation}
    \label{eq:lower_bound}
    \frac{\varepsilon}{2 - \varepsilon} p(1)^2 = \frac{1}{4} \bigl(\alpha^{q + 1/2} - \alpha^{-(q + 1/2)}\bigr)^2 \geq \frac{1}{4} \left(\frac{2}{\sqrt{c_1}} - \frac{\sqrt{c_1}}{2}\right)^2 = \frac{1}{4}\left(\frac{4}{c_1} - 2 + \frac{c_1}{4}\right) \geq \frac{1-c_1}{c_1}.
\end{equation}
We have succeeded in bounding $p(1)$ from below.

\textbf{(iii) Completing the argument.} Finally we will show that
    \begin{equation*}
    \frac{\sum_{i=2}^n c_i (1 - \lambda_i) p(\lambda_i)^2}{\sum_{i=1}^n c_i p(\lambda_i)^2} \leq \varepsilon
\end{equation*}
or equivalently
    \begin{equation*}
    \sum_{i=2}^n c_i (1 - \lambda_i) p(\lambda_i)^2 - \varepsilon \sum_{i=1}^n c_i p(\lambda_i)^2 \leq 0.
\end{equation*}    
To that end, write
    \begin{equation*}
    \sum_{i=2}^n c_i (1 - \lambda_i) p(\lambda_i)^2 - \varepsilon \sum_{i=1}^n c_i p(\lambda_i)^2 = -\varepsilon c_1 p(1)^2 + \sum_{i=2}^n c_i(1-\varepsilon - \lambda_i) p(\lambda_i)^2.
\end{equation*}
The terms in the right-most sum are nonpositive if $\lambda_i > 1 - \varepsilon$.
They are nonnegative if $\lambda_i \leq 1 - \varepsilon$, but the construction of the Chebyshev polynomial ensures that
    \begin{equation*}
    (1 - \varepsilon - \lambda_i) p(\lambda_i)^2 = \frac{2-\varepsilon}{2}(1-t_i) W_q(t_i)^2 \leq 2 - \varepsilon, \qquad \text{where } t_i = \frac{2 \lambda_i + \varepsilon}{2 - \varepsilon}.
\end{equation*}
Using the lower bound estimate \cref{eq:lower_bound}, we therefore obtain the majorant
\begin{align*}
    &-\varepsilon c_1 p(1)^2 + \sum_{i=2}^n c_i(1-\varepsilon - \lambda_i) p(\lambda_i)^2 \\
    &\quad\leq - \varepsilon c_1 p(1)^2 + (2 - \varepsilon) (1 -c_1) \\
    &\quad\leq 0.
\end{align*}
This completes the proof.
\end{proof}

\subsection{Power method analysis}
\label{sec:power}

This section analyzes the power method, as presented in \cref{alg:power}.
The gap-independent convergence bound is weaker than for Lanczos and additionally assumes that $\bm{A}$ is positive semidefinite.
The gap-dependent analysis is classical \cite[Sec.~4.1.1]{Saad:2011}, but we include it here for comparison with our analysis of the truncated power method.

\begin{algorithm}[t]
\caption[Power method]{Power method}
\label{alg:power}
\raggedright
{\bfseries Input}: Hermitian matrix \(\bm{A}\in \mathbb{C}^{n \times n}\).
Starting unit-norm vector \(\bm{x}_0 \in \mathbb{C}^{n}\).
Number of iterations \(q\). \par 
{\bfseries Output}: Approximation $(\hat{\lambda}_1, \hat{\bm{v}}_1)$ to leading eigenpair $(\lambda_1, \bm{v}_1)$ of $\bm{A}$. \par
\begin{algorithmic}[1]
\FOR{$t=1, \ldots, q$}
\STATE $\bm{y}_t\gets \bm{A}\bm{x}_{t-1}$.
\IF{$\bm{y}_t=\bm{0}$}
\STATE {\bfseries Return}: $(0,\bm{x}_{t-1})$.
\ENDIF
\STATE $\bm{x}_t\gets \bm{y}_t/\lVert\bm{y}_t\rVert$.
\ENDFOR
\STATE {\bfseries Return}: $(\hat{\lambda}_1, \hat{\bm{v}}_1) = (\bm{x}_q^\ast \bm{A} \bm{x}_q, \bm{x}_q)$.
\end{algorithmic}
\end{algorithm}

\begin{proposition}[Power method convergence] \label{prop:gapped_power}
Suppose $\bm{A} \in \mathbb{C}^{n\times n}$ is a Hermitian matrix with largest-magnitude eigenpair $(\lambda_1, \bm{v}_1)$ and second-largest-magnitude eigenvalue $\lambda_2$.
Let $\bm{x}_0 \in \mathbb{C}^n$ be a unit-norm starting vector with $\bm{v}_1^* \bm{x}_0 \neq 0$, and fix a target accuracy $\varepsilon \in (0, 1)$.
\begin{enumerate}
\item \textbf{Gapless case.} If $\bm{A}$ is positive-semidefinite, then for any number of iterations satisfying
\begin{equation*}
    q \geq \frac{1}{2 \varepsilon}\log\biggl(\frac{1}{|\bm{v}_1^* \bm{x}_0 \vert^2}\biggr),
\end{equation*}
\cref{alg:power} returns an estimate $\hat{\lambda}_1$ satisfying $|\lambda_1 - \hat{\lambda}_1| \leq \varepsilon|\lambda_1|$.
\item \textbf{Gapped case.} If $|\lambda_2| < |\lambda_1|$, then for any number of iterations satisfying
\begin{equation*}
    q \geq \frac{1}{2(1 - |\lambda_2/\lambda_1|)}\log\biggl(\frac{2}{\varepsilon\vert \bm{v}_1^\ast \bm{x}_0\vert^2}\biggr),
\end{equation*}
\cref{alg:power} returns an estimate $(\hat{\lambda}_1, \hat{\bm{v}}_1)$ satisfying
\begin{equation*}
    |\bm{v}_1^* \hat{\bm{v}}_1|^2 \geq 1 - \varepsilon/2 \quad \text{and therefore} \quad |\lambda_1 - \hat{\lambda}_1| \leq \varepsilon|\lambda_1|.
\end{equation*}
\end{enumerate}
\end{proposition}
\begin{proof}
The method is exact for $\lambda_1 = 0$, so we focus on the case $\lambda_1 \neq 0$.
Scaling $\bm{A}$ by $1/\lambda_1$ does not change the relative eigenvalue error, so we assume $\lambda_1=1$.
Take an eigendecomposition $\bm{A} = \sum_{i=1}^n \lambda_i \bm{v}_i \bm{v}_i^\ast$ and introduce coefficients $c_i = \vert \bm{v}_i^\ast \bm{x}_0\vert^2$ for $i = 1,\ldots, n$.
Let $(\hat{\lambda}_1, \hat{\bm{v}}_1)$ be the eigenpair returned after $q$ iterations of the power method.

\textbf{Gapless case.} Define $M_0=1$ and
    \begin{equation*}
    M_t=\sum_{i=1}^n c_i \lambda_i^t, \qquad t>0.
\end{equation*}
By H\"older's inequality, $M_{2q}\leq M_{2q+1}^{2q/(2q+1)}$.
Consequently,
    \begin{equation*}
    \hat{\lambda}_1 =\frac{M_{2q+1}}{M_{2q}} \geq M_{2q+1}^{1/(2q+1)} \geq c_1^{1/(2q+1)}.
\end{equation*}
Using $1-\mathrm{e}^{-x}\leq x$, we obtain
    \begin{equation*}
    0\leq 1-\hat{\lambda}_1 \leq \frac{\log(1/c_1)}{2q+1}.
\end{equation*}
This proves the claimed bound.

\textbf{Gapped case.} We make the calculation
    \begin{equation*}
    1 - |\bm{v}_1^* \bm{x}_q|^2 = \frac{\sum_{i=2}^n c_i \lambda_i^{2q}}{\sum_{i=1}^n c_i \lambda_i^{2q}} \leq \frac{\lambda_2^{2q}}{c_1} \leq \frac{1}{c_1} \mathrm{e}^{-2q(1 - |\lambda_2|)}.
\end{equation*}
We can take
    \begin{equation*}
    q \geq \frac{1}{2(1 - |\lambda_2|)}\log\biggl(\frac{2}{\varepsilon c_1}\biggr)
\end{equation*}
to guarantee $|\bm{v}_1^* \bm{x}_q|^2 \geq 1 - \varepsilon/2$.
Expansion in an orthonormal eigenbasis then gives
    \begin{equation*}
    1-\bm{x}_q^*\bm{A}\bm{x}_q = \sum_{i=2}^n (1-\lambda_i) |\bm{v}_i^*\bm{x}_q|^2 \leq 2\left(1-|\bm{v}_1^*\bm{x}_q|^2\right) \le \varepsilon.
\end{equation*}
This completes the proof.
\end{proof}

\section{Analysis of sparsified power method}
\label{sec:analysis_sparsified_power}

Here we analyze the truncated power method, which was presented in \cref{alg:sparsified}.
Our analysis balances the progress from the power step against the possible loss from truncation.
The technical crux is controlling the truncation error in terms of the best $k$-term $\ell^2$ approximation error for the leading eigenvector $\bm{v}_1$ and the squared overlap of the current iterate with $\bm{v}_1$.

\begin{proof}[Proof of \cref{thm:main_thm}]
Scaling $\bm{A}$ by $1/\lambda_1$ does not change the eigenvector error or relative eigenvalue error, so we assume $\lambda_1=1$.
We prove the result in four steps.

\paragraph{Step 1. Effect of truncation.}
We will derive a lower bound on $|\bm{v}_1^*\bm{\psi}_m(\bm{y})|^2/ \lVert\bm{\psi}_m(\bm{y})\rVert^2$, valid for any unit-norm vector $\bm{y}$.

We assume for the moment that $m<n$.
Let $\mathsf{K}$ contain the $k$ largest-magnitude coordinates of $\bm{v}_1$, and let $\mathsf{S}$ contain the $m$ coordinates retained in $\bm{y}$.
The ordering of the retained coordinates gives
\begin{equation*}
    \lVert\bm{y}(\mathsf{K}\setminus\mathsf{S})\rVert^2 \le \frac{|\mathsf{K}\setminus\mathsf{S}|}{|\mathsf{S}\setminus\mathsf{K}|} \lVert\bm{y}(\mathsf{S}\setminus\mathsf{K})\rVert^2 \le \frac{k}{m} \lVert\bm{y}(\mathsf{S}\setminus\mathsf{K})\rVert^2.
\end{equation*}
For the second inequality, we used
\begin{equation*}
    \frac{|\mathsf{K}\setminus\mathsf{S}|}{|\mathsf{S}\setminus\mathsf{K}|} = \frac{k-|\mathsf{K}\cap\mathsf{S}|}{m-|\mathsf{K}\cap\mathsf{S}|} \le \frac{k}{m}.
\end{equation*}
Next, choose the phase of $\bm{v}_1$ so that $\bm{v}_1^*\bm{y}=|\bm{v}_1^*\bm{y}|$.
Then
\begin{equation*}
    \lVert\bm{y}-|\bm{v}_1^*\bm{y}|\bm{v}_1\rVert = \lVert\bm{v}_1-|\bm{v}_1^*\bm{y}|\bm{y}\rVert = \sqrt{1-|\bm{v}_1^*\bm{y}|^2}.
\end{equation*}
Since $\mathsf{S}\setminus\mathsf{K}\subseteq\mathsf{K}^c$, the triangle inequality gives
\begin{equation*}
    \lVert\bm{y}(\mathsf{S}\setminus\mathsf{K})\rVert \le |\bm{v}_1^*\bm{y}| \lVert\bm{v}_1(\mathsf{S}\setminus\mathsf{K})\rVert + \lVert\bm{y}-|\bm{v}_1^*\bm{y}|\bm{v}_1\rVert \le |\bm{v}_1^*\bm{y}| \lVert\bm{v}_1(\mathsf{K}^c)\rVert + \sqrt{1-|\bm{v}_1^*\bm{y}|^2}.
\end{equation*}
For compact notation, write
\begin{equation*}
    \Delta_k = \min_{\operatorname{nnz}(\bm{w})\le k} \lVert\bm{v}_1-\bm{w}\rVert = \lVert\bm{v}_1(\mathsf{K}^c)\rVert \quad \text{and} \quad \eta = \sqrt{\frac{k}{m}}.
\end{equation*}
Combining the preceding bounds gives
\begin{equation}
    \label{eq:simple_bound}
    \lVert\bm{y}(\mathsf{K}\setminus\mathsf{S})\rVert \le \eta \left( \sqrt{1-|\bm{v}_1^*\bm{y}|^2} + |\bm{v}_1^*\bm{y}|\Delta_k \right).
\end{equation}

Now introduce $\bm{w}=\bm{y}-\bm{\psi}_m(\bm{y})$, so $\bm{w}$ equals $\bm{y}$ on $\mathsf{S}^c$ and equals zero elsewhere.
Since $\bm{w}$ and $\bm{\psi}_m(\bm{y})$ have disjoint supports,
\begin{align*}
    1 - \frac{|\bm{v}_1^*\bm{\psi}_m(\bm{y})|^2}{\lVert\bm{\psi}_m(\bm{y})\rVert^2} &= \min_{\gamma\in\mathbb{C}} \lVert\bm{v}_1-\gamma\bm{\psi}_m(\bm{y})\rVert^2 \\
    &\le \lVert\bm{v}_1-|\bm{v}_1^*\bm{y}|\bm{\psi}_m(\bm{y})\rVert^2 \\
    &= \left\lVert \bm{v}_1 - |\bm{v}_1^*\bm{y}|\bm{y} + |\bm{v}_1^*\bm{y}|\bm{w} \right\rVert^2 \\
    &= 1-|\bm{v}_1^*\bm{y}|^2 + 2|\bm{v}_1^*\bm{y}| \operatorname{Re}\left[\bm{v}_1^*\bm{w}\right] - |\bm{v}_1^*\bm{y}|^2\lVert\bm{w}\rVert^2.
\end{align*}
We bound the last two terms separately on $\mathsf{K}^c$ and $\mathsf{K}$.
On $\mathsf{K}^c$, we use
\begin{align*}
    0 &\le \left\lVert \bm{v}_1(\mathsf{K}^c) -|\bm{v}_1^*\bm{y}|\bm{w}(\mathsf{K}^c) \right\rVert^2 \\
    &= \lVert\bm{v}_1(\mathsf{K}^c)\rVert^2 - 2|\bm{v}_1^*\bm{y}| \operatorname{Re}\left[ \bm{v}_1(\mathsf{K}^c)^* \bm{w}(\mathsf{K}^c) \right] + |\bm{v}_1^*\bm{y}|^2 \lVert\bm{w}(\mathsf{K}^c)\rVert^2.
\end{align*}
Since $\lVert\bm{v}_1(\mathsf{K}^c)\rVert=\Delta_k$, rearranging gives
\begin{equation*}
    2|\bm{v}_1^*\bm{y}| \operatorname{Re}\left[ \bm{v}_1(\mathsf{K}^c)^* \bm{w}(\mathsf{K}^c) \right] - |\bm{v}_1^*\bm{y}|^2 \lVert\bm{w}(\mathsf{K}^c)\rVert^2 \le \Delta_k^2.
\end{equation*}
On $\mathsf{K}$, the vector $\bm{w}$ equals $\bm{y}$ on $\mathsf{K}\setminus\mathsf{S}$ and vanishes elsewhere.
Thus,
\begin{align*}
    & 2|\bm{v}_1^*\bm{y}| \operatorname{Re}\left[ \bm{v}_1(\mathsf{K})^* \bm{w}(\mathsf{K}) \right] - |\bm{v}_1^*\bm{y}|^2 \lVert\bm{w}(\mathsf{K})\rVert^2 \\
    &\quad= 2|\bm{v}_1^*\bm{y}| \operatorname{Re}\left[ \bm{v}_1(\mathsf{K}\setminus\mathsf{S})^* \bm{y}(\mathsf{K}\setminus\mathsf{S}) \right] - |\bm{v}_1^*\bm{y}|^2 \lVert\bm{y}(\mathsf{K}\setminus\mathsf{S})\rVert^2 \\
    &\quad= 2|\bm{v}_1^*\bm{y}| \operatorname{Re}\left[ \left( \bm{v}_1(\mathsf{K}\setminus\mathsf{S}) - |\bm{v}_1^*\bm{y}| \bm{y}(\mathsf{K}\setminus\mathsf{S}) \right)^* \bm{y}(\mathsf{K}\setminus\mathsf{S}) \right] + |\bm{v}_1^*\bm{y}|^2 \lVert\bm{y}(\mathsf{K}\setminus\mathsf{S})\rVert^2 \\
    &\quad\le 2|\bm{v}_1^*\bm{y}| \sqrt{1-|\bm{v}_1^*\bm{y}|^2} \lVert\bm{y}(\mathsf{K}\setminus\mathsf{S})\rVert + |\bm{v}_1^*\bm{y}|^2 \lVert\bm{y}(\mathsf{K}\setminus\mathsf{S})\rVert^2.
\end{align*}
The last inequality uses Cauchy--Schwarz and $\lVert\bm{v}_1-|\bm{v}_1^*\bm{y}|\bm{y}\rVert =\sqrt{1-|\bm{v}_1^*\bm{y}|^2}$.
Applying \cref{eq:simple_bound}, we obtain
\begin{align*}
    & 2|\bm{v}_1^*\bm{y}| \sqrt{1-|\bm{v}_1^*\bm{y}|^2} \lVert\bm{y}(\mathsf{K}\setminus\mathsf{S})\rVert + |\bm{v}_1^*\bm{y}|^2 \lVert\bm{y}(\mathsf{K}\setminus\mathsf{S})\rVert^2 \\
    &\quad\le 2\eta|\bm{v}_1^*\bm{y}| \left(1-|\bm{v}_1^*\bm{y}|^2\right) + 2\eta|\bm{v}_1^*\bm{y}|^2 \sqrt{1-|\bm{v}_1^*\bm{y}|^2} \Delta_k \\
    &\qquad\quad + \eta^2|\bm{v}_1^*\bm{y}|^2 \left( \sqrt{1-|\bm{v}_1^*\bm{y}|^2} + |\bm{v}_1^*\bm{y}|\Delta_k \right)^2.
\end{align*}
To bound the cross terms, we use
\begin{equation*}
    2\eta|\bm{v}_1^*\bm{y}|^2 \sqrt{1-|\bm{v}_1^*\bm{y}|^2} \Delta_k \le 3\eta^2|\bm{v}_1^*\bm{y}|^2 \left(1-|\bm{v}_1^*\bm{y}|^2\right) + \frac{1}{3}|\bm{v}_1^*\bm{y}|^2\Delta_k^2
\end{equation*}
due to $2ab = 2(\sqrt{3} a)(b/\sqrt{3}) \le 3a^2+b^2/3$, and
\begin{equation*}
    \left( \sqrt{1-|\bm{v}_1^*\bm{y}|^2} + |\bm{v}_1^*\bm{y}|\Delta_k \right)^2 \le 2\left(1-|\bm{v}_1^*\bm{y}|^2\right) + 2|\bm{v}_1^*\bm{y}|^2\Delta_k^2.
\end{equation*}
Combining these estimates gives
\begin{align*}
    1 - \frac{|\bm{v}_1^*\bm{\psi}_m(\bm{y})|^2}{\lVert\bm{\psi}_m(\bm{y})\rVert^2} &\le \left( 1 + 2\eta|\bm{v}_1^*\bm{y}| + 5\eta^2|\bm{v}_1^*\bm{y}|^2 \right) \left(1-|\bm{v}_1^*\bm{y}|^2\right) + \left( 1 + \frac{1}{3}|\bm{v}_1^*\bm{y}|^2 + 2\eta^2|\bm{v}_1^*\bm{y}|^4 \right) \Delta_k^2 \\
    &\leq \left( 1 + 2\eta|\bm{v}_1^*\bm{y}| + 5\eta^2|\bm{v}_1^*\bm{y}|^2 \right) \left(1-|\bm{v}_1^*\bm{y}|^2\right) + \frac{5}{3} \Delta_k^2,
\end{align*}
where we used $|\bm{v}_1^*\bm{y}|\le1$ and $\eta^2\le\frac{1}{6}$, as guaranteed by the oversampling assumption.

We have established
\begin{equation}
    \label{eq:normalized_truncation}
    \frac{|\bm{v}_1^*\bm{\psi}_m(\bm{y})|^2}{\lVert\bm{\psi}_m(\bm{y})\rVert^2} \ge |\bm{v}_1^*\bm{y}|^2 - \left( 2\eta|\bm{v}_1^*\bm{y}| + 5\eta^2|\bm{v}_1^*\bm{y}|^2 \right) \left(1-|\bm{v}_1^*\bm{y}|^2\right) - \frac{5}{3}\Delta_k^2
\end{equation}
assuming $m < n$.
Additionally, \cref{eq:normalized_truncation} holds when $m\ge n$, because then there is no truncation.

\paragraph{Step 2. Progress from one iteration.}
We consider an iterate $\bm{x}_t$ satisfying
\begin{equation}
    \label{eq:spm_overlap_region}
    |\bm{v}_1^*\bm{x}_t|^2 \ge \min\left\{ |\bm{v}_1^*\bm{x}_0|^2, 1-\frac{\varepsilon}{2} \right\}
\end{equation}
and we will assess the progress from one iteration.

Let $\bm{y}_t$ be the normalized matrix-vector product, $\bm{y}_t = \bm{A} \bm{x}_t / \lVert \bm{A} \bm{x}_t \rVert$ and write
\begin{equation*}
    \lVert\bm{A}\bm{x}_t\rVert^2 = |\bm{v}_1^*\bm{x}_t|^2 + \sigma^2 \left(1-|\bm{v}_1^*\bm{x}_t|^2\right).
\end{equation*}
Here $\sigma$ is the norm amplification of the component orthogonal to $\bm{v}_1$; we take $\sigma=0$ when $|\bm{v}_1^*\bm{x}_t|=1$.
Expansion in an eigenbasis gives $0\le\sigma\le|\lambda_2|$.
Consequently, $\sigma \le \lVert\bm{A}\bm{x}_t\rVert \le1$, and
\begin{equation*}
    |\bm{v}_1^*\bm{y}_t|^2 = \frac{|\bm{v}_1^*\bm{x}_t|^2}{\lVert \bm{A} \bm{x}_t \rVert^2} \quad \text{and} \quad 1-|\bm{v}_1^*\bm{y}_t|^2 = \frac{\sigma^2 \left(1-|\bm{v}_1^*\bm{x}_t|^2\right)}{\lVert \bm{A} \bm{x}_t \rVert^2}.
\end{equation*}
Applying \cref{eq:normalized_truncation} to $\bm{y}_t$ and subtracting the current squared overlap gives
\begin{align*}
    & |\bm{v}_1^*\bm{x}_{t+1}|^2 - |\bm{v}_1^*\bm{x}_t|^2 + \frac{5}{3}\Delta_k^2 \\
    & \quad \geq |\bm{v}_1^*\bm{y}_t|^2 - |\bm{v}_1^*\bm{x}_t|^2 - \left(2 \eta |\bm{v}_1^* \bm{y}_t| + 5 \eta^2 |\bm{v}_1^* \bm{y}_t|^2\right) \left(1 - |\bm{v}_1^* \bm{y}_t|^2\right) \\
    & \quad = \frac{|\bm{v}_1^*\bm{x}_t|^2(1 - |\bm{v}_1^* \bm{x}_t|^2)(1-\sigma^2)}{\lVert \bm{A} \bm{x}_t \rVert^2} - \left(2 \eta |\bm{v}_1^* \bm{y}_t| + 5 \eta^2 |\bm{v}_1^* \bm{y}_t|^2\right) \frac{\sigma^2 \left(1-|\bm{v}_1^*\bm{x}_t|^2\right)}{\lVert \bm{A} \bm{x}_t \rVert^2} \\
    &\quad = \frac{ |\bm{v}_1^*\bm{x}_t|^2 \left(1-|\bm{v}_1^*\bm{x}_t|^2\right) }{\lVert \bm{A} \bm{x}_t \rVert^2} \left( 1-\sigma^2 - \frac{2\eta\sigma^2}{|\bm{v}_1^*\bm{x}_t| \lVert \bm{A} \bm{x}_t \rVert} - \frac{5\eta^2\sigma^2}{\lVert \bm{A} \bm{x}_t \rVert^2} \right).
\end{align*}
The assumption \cref{eq:spm_overlap_region} guarantees that $|\bm{v}_1^*\bm{x}_t|^2 \geq |\bm{v}_1^*\bm{x}_0|^2/2$, and the oversampling assumption then implies
\begin{equation*}
    \eta^2 \le \frac{ (1-|\lambda_2|)^2 |\bm{v}_1^*\bm{x}_0|^2 }{10} \le \frac{ (1-|\lambda_2|)^2 |\bm{v}_1^*\bm{x}_t|^2 }{5}.
\end{equation*}
In particular, we can derive the inequalities
\begin{equation*}
    2\eta \le (1-|\lambda_2|) |\bm{v}_1^*\bm{x}_t| \quad \text{and} \quad 5\eta^2 \le (1-|\lambda_2|)^2 |\bm{v}_1^*\bm{x}_t|^2.
\end{equation*}
Combining these estimates and using $1-\sigma \geq 1-|\lambda_2|$ gives
\begin{align*}
    &|\bm{v}_1^*\bm{x}_{t+1}|^2 - |\bm{v}_1^*\bm{x}_t|^2 + \frac{5}{3}\Delta_k^2 \\
    &\quad\geq \bigl(1-|\lambda_2|\bigr)\frac{ |\bm{v}_1^*\bm{x}_t|^2 \left(1-|\bm{v}_1^*\bm{x}_t|^2\right) }{\lVert \bm{A} \bm{x}_t \rVert^2} \\
    &\qquad\quad\times \left( 1 + \sigma - \frac{\sigma^2}{\lVert \bm{A} \bm{x}_t \rVert} - \frac{(1-\sigma) \sigma^2 |\bm{v}_1^* \bm{x}_t|^2}{\lVert \bm{A} \bm{x}_t \rVert^2} \right).
\end{align*}
To simplify this expression, first suppose $\sigma>0$.
We observe
\begin{equation*}
    \frac{(1-\sigma^2)|\bm{v}_1^*\bm{x}_t|^2}{\lVert \bm{A} \bm{x}_t \rVert^2} = 1-\frac{\sigma^2}{\lVert \bm{A} \bm{x}_t \rVert^2} \le 2\left( 1-\frac{\sigma}{\lVert \bm{A} \bm{x}_t \rVert} \right) \le \frac{1 + \sigma}{\sigma} \left(1-\frac{\sigma}{\lVert \bm{A} \bm{x}_t \rVert}\right).
\end{equation*}
The first inequality uses $1-t^2\le2(1-t)$ for $0\le t\le1$.
The second uses $2\sigma\le1+\sigma$.
Multiplying both sides by $\sigma^2/(1+\sigma)$ yields
\begin{equation*}
    \frac{(1-\sigma) \sigma^2|\bm{v}_1^*\bm{x}_t|^2}{\lVert \bm{A} \bm{x}_t \rVert^2} \le \sigma - \frac{\sigma^2}{\lVert \bm{A} \bm{x}_t \rVert}.
\end{equation*}
When $\sigma=0$, this bound holds because both sides are zero.
Thus, we arrive at
\begin{align*}
    |\bm{v}_1^*\bm{x}_{t+1}|^2 - |\bm{v}_1^*\bm{x}_t|^2 + \frac{5}{3}\Delta_k^2 &\geq \bigl(1-|\lambda_2|\bigr)\frac{ |\bm{v}_1^*\bm{x}_t|^2 \left(1-|\bm{v}_1^*\bm{x}_t|^2\right) }{\lVert \bm{A} \bm{x}_t \rVert^2} \\
    &\geq \bigl(1-|\lambda_2|\bigr) |\bm{v}_1^*\bm{x}_t|^2 \left(1-|\bm{v}_1^*\bm{x}_t|^2\right).
\end{align*}
Finally, the compressibility assumption gives
\begin{equation*}
    \frac{5}{3}\Delta_k^2 \le \frac{(1-|\lambda_2|)\varepsilon}{6}.
\end{equation*}
We have therefore proved the recurrence
\begin{equation}
    \label{eq:spm_overlap_recursion}
    |\bm{v}_1^*\bm{x}_{t+1}|^2 \ge |\bm{v}_1^*\bm{x}_t|^2 + (1-|\lambda_2|) \left( |\bm{v}_1^*\bm{x}_t|^2 \left(1-|\bm{v}_1^*\bm{x}_t|^2\right) - \frac{\varepsilon}{6} \right),
\end{equation}
valid whenever \cref{eq:spm_overlap_region} holds.

\paragraph{Step 3. Reaching and preserving the target.}
Now we put the one-step recursion into action.

Suppose first that
\begin{equation*}
    |\bm{v}_1^*\bm{x}_0|^2 \le |\bm{v}_1^*\bm{x}_t|^2 < 1-\frac{\varepsilon}{2}
\end{equation*}
and recall that $\varepsilon < |\bm{v}_1^*\bm{x}_0|^2 \leq |\bm{v}_1^*\bm{x}_t|^2$.
In the case $|\bm{v}_1^*\bm{x}_t|^2 \le 2/3$, we see that the squared overlap is larger than $\varepsilon$, and its complement is at least $1/3$, so
\begin{equation*}
    |\bm{v}_1^*\bm{x}_t|^2 \left(1-|\bm{v}_1^*\bm{x}_t|^2\right) \geq \frac{\varepsilon}{3}.
\end{equation*}
Otherwise, $|\bm{v}_1^*\bm{x}_t|^2 > 2/3$ and the complement is larger than $\varepsilon/2$, which gives the same bound.
Substituting this estimate into \cref{eq:spm_overlap_recursion} gives
\begin{equation}
    \label{eq:spm_overlap_growth}
    |\bm{v}_1^*\bm{x}_{t+1}|^2 \ge |\bm{v}_1^*\bm{x}_t|^2 + \frac{1-|\lambda_2|}{2} |\bm{v}_1^*\bm{x}_t|^2 \left(1-|\bm{v}_1^*\bm{x}_t|^2\right).
\end{equation}
Thus, until the target is reached, the squared overlap increases and remains at least its initial value.

Next suppose $|\bm{v}_1^*\bm{x}_t|^2\ge1-\varepsilon/2$.
As a function of the current squared overlap, the right-hand side of \cref{eq:spm_overlap_recursion} has derivative at least $|\lambda_2|\ge0$ on $[0,1]$.
To verify that the target is preserved, it therefore suffices to check the boundary $|\bm{v}_1^*\bm{x}_t|^2=1-\varepsilon/2$.
At this boundary, the increment in the recurrence is
\begin{equation*}
    (1-|\lambda_2|) \left( \frac{\varepsilon}{2} - \frac{\varepsilon^2}{4} - \frac{\varepsilon}{6} \right) = (1-|\lambda_2|) \frac{\varepsilon(4-3\varepsilon)}{12} > 0,
\end{equation*}
because $\varepsilon<1$.
We verify that $|\bm{v}_1^* \bm{x}_{t+1}|^2 \geq 1 - \varepsilon/2$ so the target is preserved.

\paragraph{Step 4. Iteration count.}
Now suppose the number of iterations is
\begin{equation*}
    q \geq \frac{5}{2(1 - |\lambda_2|)} \log\left(\frac{2}{\varepsilon |\bm{v}_1^* \bm{x}_0|^2}\right).
\end{equation*}
If the target is reached by iteration $q$, Step 3 shows that it is preserved.
Otherwise, the squared overlap remains below $1-\varepsilon/2$ through iteration $q$, so we may apply \cref{eq:spm_overlap_growth} at each iteration.
Taking ratios gives
\begin{equation*}
    \frac{ |\bm{v}_1^*\bm{x}_{t+1}|^2 }{ 1-|\bm{v}_1^*\bm{x}_{t+1}|^2 } \ge \frac{ 1 + \frac{1-|\lambda_2|}{2} \left(1-|\bm{v}_1^*\bm{x}_t|^2\right) }{ 1 - \frac{1-|\lambda_2|}{2} |\bm{v}_1^*\bm{x}_t|^2 } \frac{ |\bm{v}_1^*\bm{x}_t|^2 }{ 1-|\bm{v}_1^*\bm{x}_t|^2 } \ge \left( 1+\frac{1-|\lambda_2|}{2} \right) \frac{ |\bm{v}_1^*\bm{x}_t|^2 }{ 1-|\bm{v}_1^*\bm{x}_t|^2 }.
\end{equation*}
Thus the ratio of squared overlap to squared error grows by a fixed factor at each iteration.
Using the inequality $\log(1+s)\geq 2s/(2+s)$ for $s\ge0$, we find that
\begin{equation*}
    \log\left( 1+\frac{1-|\lambda_2|}{2} \right) \ge \frac{2(1-|\lambda_2|)}{5-|\lambda_2|} \ge \frac{2}{5}(1-|\lambda_2|).
\end{equation*}
Consequently,
\begin{equation*}
    \left( 1+\frac{1-|\lambda_2|}{2} \right)^q \geq \exp\left(\frac{2q}{5}(1-|\lambda_2|)\right) \geq \frac{2}{\varepsilon |\bm{v}_1^* \bm{x}_0|^2}.
\end{equation*}
After $q$ iterations, the squared error satisfies
\begin{equation*}
    \frac{ 1 }{ 1-|\bm{v}_1^*\bm{x}_q|^2 } \geq \frac{ |\bm{v}_1^*\bm{x}_q|^2 }{ 1-|\bm{v}_1^*\bm{x}_q|^2 } \ge \left( 1+\frac{1-|\lambda_2|}{2} \right)^q \frac{ |\bm{v}_1^*\bm{x}_0|^2 }{ 1-|\bm{v}_1^*\bm{x}_0|^2 } \ge \frac{2}{\varepsilon}.
\end{equation*}
This contradicts the assumption that the target has not been reached.
Therefore, the target is reached by iteration $q$ and preserved thereafter.
We confirm that
\begin{equation*}
    |\bm{v}_1^*\bm{x}_q|^2 \ge 1-\frac{\varepsilon}{2}.
\end{equation*}
Expansion in an orthonormal eigenbasis then gives
\begin{equation*}
    1-\bm{x}_q^*\bm{A}\bm{x}_q = \sum_{j=2}^n (1-\lambda_j) |\bm{v}_j^*\bm{x}_q|^2 \leq 2\left(1-|\bm{v}_1^*\bm{x}_q|^2\right) \le \varepsilon.
\end{equation*}
This completes the proof.
\end{proof}

\section{Analysis of sparsified submatrix selection}
\label{sec:analysis_sparsified_submatrix}

Here we prove \cref{prop:greedy_negative_thm}.
The worst-case construction is a symmetric block diagonal matrix with a $3 \times 3$ tridiagonal block followed by a $1 \times 1$ block.
The spectral norms of the blocks are $1$ and $1-r$, respectively, yet the sparsified submatrix selection method mistakenly chooses the second block for its support.
This happens because the initialization vector has nonzero entries at coordinates 1 and 4, the first iteration leads to an active set containing indices 1, 2, and 4, and the spectral norm of the leading $2 \times 2$ matrix is $1-q$, where $q$ is slightly larger than $r$.
The exact construction depends on $q$ in a slightly complicated way, which comes from making $|\bm{v}_1(1)|^2$ as large as possible subject to the spectral constraints.

\begin{proof}[Proof of \cref{prop:greedy_negative_thm}]

The proof proceeds in three steps.

\emph{Step 1.
Build the starting vector and target matrix.} Choose $q\in(r,1)$ sufficiently close to $r$ so that
\begin{equation*}
    \delta < \frac{2(1-q)^3}{(1+q)^2(2-q)} < \frac{2(1-r)^3}{(1+r)^2(2-r)}
\end{equation*}
and define a starting vector
\begin{equation*}
    \bm{x}_0 = \begin{bmatrix}
    \alpha^{1/2} \\
    0 \\
    0 \\
    (1-\alpha)^{1/2}
    \end{bmatrix}, \qquad \alpha = \frac{(1+q)^2(2-q)}{2(1-q)^3} \delta.
\end{equation*}
Thus, $0 < \alpha < 1$ and $\bm{x}_0$ is supported on coordinates 1 and 4.
Introduce a rotation
\begin{equation*}
    \bm{R} = \begin{bmatrix}
    c & -s \\
    s & c
    \end{bmatrix}, \qquad c = \sqrt{\frac{3-q}{3+3q}}, \qquad s = \sqrt{\frac{4q}{3+3q}}
\end{equation*}
and define a $2 \times 2$ matrix,
\begin{equation*}
    \bm{C} = \begin{bmatrix}
    \bm{C}(1, 1) & \bm{C}(1, 2) \\
    \bm{C}(2, 1) & \bm{C}(2, 2)
    \end{bmatrix} = \bm{R} \begin{bmatrix}
    1-q & 0 \\
    0 & (q-1)/2
    \end{bmatrix} \bm{R}^\top.
\end{equation*}
Last, construct the target matrix
\begin{equation*}
    \bm{A} = \begin{bmatrix}
    \bm{C}(1, 1) & \bm{C}(1, 2) & 0 & 0 \\
    \bm{C}(2, 1) & \bm{C}(2, 2) & d & 0 \\
    0 & d & q & 0 \\
    0 & 0 & 0 & 1-r
    \end{bmatrix}, \qquad d = \sqrt{\frac{1 - q^2}{2}}.
\end{equation*}

\emph{Step 2.
Check the spectral properties.} We carefully check the spectral properties.
By rotating the upper left $3 \times 3$ block, we expose a matrix
\begin{equation*}
    \bm{H} = \begin{bmatrix}
    c & s & 0 \\
    -s & c & 0 \\
    0 & 0 & 1
    \end{bmatrix} \begin{bmatrix}
    \bm{C}(1, 1) & \bm{C}(1, 2) & 0 \\
    \bm{C}(2, 1) & \bm{C}(2, 2) & d \\
    0 & d & q
    \end{bmatrix} \begin{bmatrix}
    c & -s & 0 \\
    s & c & 0 \\
    0 & 0 & 1
    \end{bmatrix} = \begin{bmatrix}
    1-q & 0 & ds \\
    0 & (q-1)/2 & dc \\
    ds & dc & q
    \end{bmatrix}.
\end{equation*}
A direct calculation with determinants shows that
\begin{equation*}
    \det(z \mathbf{I} - \bm{H}) = (z-1)(z-(1-q)/2)(z-(q-1)),
\end{equation*}
so the eigenvalues are $1, (1-q)/2, q-1$.
The remaining eigenvalue of $\bm{A}$ is $1-r$.
Since $r<q<1$, the two largest-magnitude eigenvalues of $\bm{A}$ are therefore $1$ and $1-r$.
Another calculation shows the leading eigenvector of $\bm{H}$ is
\begin{equation*}
    \frac{1}{\sqrt{3(1+q)(2-q)}} \begin{bmatrix}
    \sqrt{2(1-q)(3-q)} \\
    \sqrt{2q(1-q)} \\
    \sqrt{3q(3-q)}
    \end{bmatrix}.
\end{equation*}
Last, rotating back to the coordinate frame of $\bm{A}$, the first coordinate of the leading eigenvector is
\begin{equation*}
    \bm{v}_1(1) = \frac{1}{\sqrt{3(1+q)(2-q)}} \begin{bmatrix}
    c & -s
    \end{bmatrix} \begin{bmatrix}
    \sqrt{2(1-q)(3-q)} \\
    \sqrt{2q(1-q)}
    \end{bmatrix} = \frac{\sqrt{2(1-q)^3}}{\sqrt{(1+q)^2(2-q)}}.
\end{equation*}
This confirms that the initial squared overlap is $|\bm{v}_1^*\bm{x}_0|^2 = \delta$, as desired.
The second and third coordinates are also positive, while the fourth coordinate is zero.
Thus, $\operatorname{supp}(\bm{v}_1)=\{1,2,3\}$.

\emph{Step 3.
Analyze the algorithm.} Since $\bm{C}(1,2)>0$ and the starting vector is supported on coordinates 1 and 4, the first active set is $\mathsf{S}_1=\{1,2,4\}$.
The selected principal submatrix consists of $\bm{C}$, whose spectral norm is $1-q$, and the scalar block $1-r$.
Since $1-r>1-q$, the first eigenvector estimate is $\bm{x}_1=\pm\bm{e}_4$.
Column 4 has no off-diagonal nonzeros, so every subsequent active set is $\{4\}$.
Therefore, $\bm{x}_t=\pm\bm{e}_4$ for all $t\ge1$, and every subsequent iterate has zero overlap with $\bm{v}_1$.
This completes the proof.
\end{proof}

\section{Conclusions and outlook} \label{sec:outlook}

We considered whether it is possible to estimate a leading eigenpair of a sparse $n \times n$ Hermitian matrix in time that is independent of $n$ and made two main contributions.
First, we showed that the truncated power method retains a convergence rate comparable to the standard power method when the retained sparsity is chosen according to the target accuracy, the initial overlap, the relative spectral gap, and the compressibility of the leading eigenvector.
Second, for a sparsified submatrix selection method inspired by quantum chemistry, we issued a warning: when the initial squared overlap lies below an explicit threshold determined by the relative spectral gap, there are worst-case instances in which the method loses all overlap with the leading eigenvector after a single iteration and never recovers.

Taken as a whole, this paper identifies structural factors that govern whether dimension-independent estimation of eigenvalues and eigenvectors is possible.
Two factors---the initial overlap and the spectral gap---are already familiar from classical convergence analyses of the power and Lanczos methods \cite{KuczynskiWozniakowski:1992,Saad:2011}.
The third factor, compressibility of the leading eigenvector, is a more novel feature of our analysis.
Yet it is very natural: any method that approximates a high-dimensional vector by a sparse vector can succeed only when the target is compressible at the retained sparsity level.
Here we quantify compressibility in terms of the best $k$-term $\ell^2$ approximation error,
\begin{equation*}
    \Delta_k = \min_{\operatorname{nnz}(\bm{x}) \leq k} \lVert \bm{v}_1 - \bm{x} \rVert_2.
\end{equation*}
We expect that analogous tail quantities will play a key role in future analyses of dimension-independent linear algebra algorithms.

Our chemistry experiments show that submatrix selection methods reach a given accuracy in fewer outer iterations than the truncated power method.
Yet understanding this practical success remains an open problem.
Another open problem is whether randomized sparsification can improve the worst-case guarantees of deterministic truncation for the Hermitian eigenvalue problem.

In the broader context, scientific computing and machine learning now lead to a variety of linear algebra problems at extreme scales.
Our work has mainly focused on electronic structure theory problems, in which the many-electron basis size increases exponentially with the number of electrons \cite{szabo2012modern}.
More generally, optimization problems can have combinatorially many variables \cite{LubbeckeDesrosiers:2005}, while tensor-product discretizations of high-dimensional PDEs produce exponentially large state spaces \cite{BungartzGriebel:2004}.
Large-scale linear algebra also arises in the training and analysis of neural networks \cite{MartensGrosse:2015,GhorbaniKrishnanXiao:2019}.

Much of numerical linear algebra assumes that at least one dense vector of length $n$ can be stored, and great progress has been made within this regime.
For problems with exponential scaling, however, memory is a bottleneck and even this assumption fails.
Less attention has been devoted to the extreme-scale regime, but it cannot be avoided forever.
Our work begins to provide algorithms and useful guarantees for this less developed context.

\begin{ack}
The authors acknowledge helpful conversations with Jackie Lok and Samuel Greene.

During the preparation of this work, the authors used ChatGPT to improve the writing and mathematical results, and Claude to improve the code.
The main theorems were developed before AI tools were used, although AI-assisted revisions improved their constants or rates.
The authors independently verified all proofs and computational results and take full responsibility for the published work.

This work was supported in part through the NYU IT High Performance Computing resources, services, and staff expertise.
\end{ack}

\begin{funding}
RJW was supported by a 2026--2027 Hellman Fellowship. CM was supported by the NSF AF program under award $\#2427363$. NH and JW were supported by the NSF CAIG program under award $\#2425899$.
\end{funding}

\bibliographystyle{emss-links}
\bibliography{ref}

\end{document}